\documentclass[12pt]{article}

\usepackage[T1,T2A]{fontenc}
\usepackage[utf8]{inputenc}

\usepackage{amsfonts}\usepackage{amssymb}
\usepackage{amsmath}
\usepackage{amsthm}
\usepackage{hyperref}
\usepackage{subfiles}
\usepackage{enumitem}

\usepackage{blindtext}

\usepackage{titlesec}

\newtheorem{theorem}{Theorem}[section]
\newtheorem{proposition}[theorem]{Proposition}
\newtheorem{lemma}[theorem]{Lemma}
\newtheorem{corollary}[theorem]{Corollary}
\newtheorem{definition}[theorem]{Definition}
\newtheorem{remark}[theorem]{Remark}
\newtheorem{thm}{Theorem} 

\newtheorem{lem}[thm]{Lemma}

\def \le {\leqslant}
\def \ge {\geqslant}

\numberwithin{equation}{section}

\usepackage{tikz} 
\usetikzlibrary{calc,arrows,decorations.pathreplacing,fadings,3d,positioning}

\makeatletter
\def\namedlabel#1#2{\begingroup
    #2%
    \def\@currentlabel{#2}%
    \phantomsection\label{#1}\endgroup
}
\makeatother

\title{Metric theory of inhomogeneous Diophantine approximations with a fixed matrix}

\author{Nikolay Moshchevitin \\
        { \small Institute für diskrete Mathematik und Geomertie, Technische Universität Wien, Wien 1040, Austria} \\
        {\small \texttt{nikolai.moshchevitin@tuwien.ac.at, moshchevitin@gmail.com}} \\
  Vasiliy Neckrasov \\
  {\small Brandeis University, Waltham, MA 02453, USA} \\
  {\small \texttt{vneckrasov@brandeis.edu}}
}

\date{}

\begin{document}

\maketitle

\begin{abstract}
    In this paper we develop the metric theory of inhomogeneous Diophantine approximation for the case of a fixed matrix. We use transference principle to connect uniform Diophantine properties of a pair $(\Theta, \pmb{\eta})$ of a matrix and a vector with the asymptotic Diophantine properties of the transposed matrix $\Theta^{\top}$, and vice versa, the asymptotic Diophantine properties of a pair $(\Theta, \pmb{\eta})$  with uniform Diophantine properties of the transposed matrix. In these setups, we prove analogues of classical statements of metric homogeneous Diophantine approximations and answer some open questions that were raised in recent works. 
\end{abstract}

\vskip+0.2cm


\vskip+0.5cm

\tableofcontents

\section{Introduction}

\subsection{Setup and notation} Throughout this paper $m$ and $n$ are positive integers. Let ${\bf M}_{n,m}$ be  the set of  all real $n \times m$ matrices.
Let $$ \Theta =\left( \begin{array}{ccc} \theta_{1,1}&\cdots&\theta_{1,m}\cr \cdots &\cdots &\cdots \cr \theta_{n ,1}&\cdots&\theta _{n,m} \end{array} \right) \in {\bf M}_{n,m},\,\,\,\,\, \pmb{\eta} = \left( \begin{array}{c} \eta_1\cr \vdots \cr \eta_n \end{array} \right) \in \mathbb{R}^n $$ be a real $n \times m$ matrix and a real vector.
We  use the notation $| \cdot |$ for the supremum norm of a vector, and the notation $|| \cdot ||_{\mathbb{Z}^n}$ will be reserved for the distance to the nearest integer vector; that is, for ${\bf v} \in \mathbb{R}^n$ we set
$$
|| {\bf v}||_{\mathbb{Z}^n} = \min\limits_{{\bf p} \in \mathbb{Z}^n} |{\bf v} - {\bf p}|.
$$
The present paper is devoted to study of inhomogeneous Diophantine Approximations, that is the small values of the  form
$$
||\Theta {\bf q} -\pmb{\eta}||_{\mathbb{Z}^n} 
$$
when $\bf q$ runs through the set of all integer $m$-dimensional vectors. 

Specifically, we will study the $\Theta$-fixed singly metric case, namely, we fix the matrix $\Theta$ and study the aforementioned form for different vectors $\pmb{\eta}$\footnote{This setup is also known as {\it twisted approximations}; see \cite{Ha12}}. We develop a comprehensive metric theory of inhomogeneous approximations for this setup, for both uniform and asymptotic (ordinary) approximation. Most of the statements in this text hold for vectors $\pmb{\eta}$ restricted to certain (but not all) affine subspaces of $\mathbb{R}^n$.

\vskip+0.3cm

We want to compare values of the form $||\Theta {\bf q} -\pmb{\eta}||_{\mathbb{Z}^n} $ with the values of certain functions in both uniform {\it ("system of inequalities is solvable for all parameters large enough")} and asymptotic {\it ("system of inequalities is solvable for an unbounded set of parameters")} setups. Thus, it is essential to introduce the notions of {\it Dirichlet} and {\it approximable} pairs relative to a certain function. We are following the standard notation from recent works, see \cite{kwpams, kwad}.

Everywhere in the text we assume that  $f(T), g(T), \, T\in \mathbb{R}_+$ 
are real valued positive functions; let us also define 
\begin{equation}\label{qq1}
f_1(T): = \frac{1}{T} \,\,\,\,\,\,\,\,\,\,\,\, \text{and, in general,} \,\,\,\,\,\,\,\,\,\,\,\, f_k(T) : = \frac{1}{T^k}.
\end{equation}

The main definitions are as follows. 

\begin{definition}\label{Dirichletdef}
We say that a matrix $\Theta$ is  {\it $f$-Dirichlet } if the system of inequalities

$$
\begin{cases}
    ||\Theta {\bf q}||_{\mathbb{Z}^n}^n \leq f(T) \\
    |{\bf q}|^m \leq T
\end{cases}
$$
has a solution ${\bf q} \in \mathbb{Z}^m \setminus \{ {\bf 0} \}$ for all $T \in \mathbb{R}_+$ large enough. We will denote the set of $n \times m$ $f$-Dirichlet real  matrices by ${\bf D}_{n,m} \left[ f \right]$. 

We will define the set of {\it inhomogeneously $g$-Dirichlet} pairs $(\Theta, \pmb{\eta})$ by the condition that the system of inequalities
$$
\begin{cases}
    ||\Theta {\bf q} - \pmb{\eta}||_{\mathbb{Z}^n}^n \leq g(T) \\
    |{\bf q}|^m \leq T
\end{cases}
$$
has a nonzero solution ${\bf q} \in \mathbb{Z}^m$ for all $T \in \mathbb{R}_+$ large enough, and denote this set $\widehat{\bf D}_{n,m} \left[ g \right]$. We utilize the notation
$$
\widehat{\bf D}_{n,m}^{\Theta} \left[ g \right] = \{ \pmb{\eta} \in \mathbb{R}^n: \,\,\,  (\Theta, \pmb{\eta}) \in \widehat{\bf D}_{n,m} \left[ g \right] \}
$$ 
for a section of such a set with fixed matrix $\Theta$.
\end{definition}


\vskip+0.3cm

\begin{definition}\label{approximabledef}

 We define the set of {\it $f$-approximable} matrices 
 {$\Theta \in {\bf M}_{n,m}$}
 as the set of such matrices that

$$
||\Theta {\bf q} ||_{\mathbb{Z}^n}^n \leq f \left(|{\bf q}|^m\right) \,\,\,\,\,\, \text{for infinitely many}\, \, {\bf q}\in \mathbb {Z}^m\setminus\{{\bf 0}\}. 
$$ 
We  denote the set of $f$-approximable $n \times m$ real matrices by ${ \bf W}_{n, m}\left[f \right]$; when it is clear from the context, we will simply write ${ \bf W}\left[f \right]$ instead.
 
 We define the set of {\it inhomogeneously $g$-approximable} pairs $(\Theta, \pmb{\eta}) \in {\bf M}_{n,m} \times \mathbb{R}^n$ as the set of such pairs that

$$
||\Theta {\bf q} - \pmb{\eta}||_{\mathbb{Z}^n}^n \leq g \left(|{\bf q}|^m\right) \,\,\,\,\,\, \text{for infinitely many } \, {\bf q}\in \mathbb {Z}^m\setminus\{{\bf 0}\}. 
$$ 
We denote the set of all  
 inhomogeneously $g$-approximable pairs $(\Theta, \pmb{\eta})$ by
 $\widehat{{ \bf W}}_{n, m}\left[ g \right]$.
For the sections of the sets $\widehat{{ \bf W}}_{n, m}\left[ g \right]$ with a fixed $\Theta$, we use the notation
$$
\widehat{{ \bf W}}^{\Theta}_{n, m}\left[ g \right] = \{ \pmb{\eta} \in \mathbb{R}^n: \,\,\,  (\Theta, \pmb{\eta}) \in \widehat{{ \bf W}}_{n, m}\left[ g \right] \}.
$$
\end{definition}

\vskip+0.3cm

In the notation above, the classical Dirichlet's theorem states that

$$
{ \bf W}_{n, m}\left[f_1 \right] = { \bf D}_{n, m}\left[f_1 \right] = \mathbb{R}^n.
$$

\vskip+0.3cm

From now on, we will assume that $f$ is a continuous decreasing function with an inverse function $f^{-1}$. To formulate our results we  consider the function $g$ defined by
\begin{equation} \label{relationg(f)}
g(T) = \frac{1}{f^{-1} \left( \frac{1}{T} \right)}.
\end{equation}
The relation between $f$ and $g$ is symmetric, that is
$$
f(T) = \frac{1}{g^{-1} \left( \frac{1}{T} \right)}.
$$

For a function $f$ and a constant $C$, we will use the notation 
{$f \circ C$} to denote the composition of multiplication by $C$ with $f$, that is, $\left( f \circ C \right) (T) = f(CT)$.

\subsection{Selected results}

We would like to start with formulating two theorems, which provide special cases of a few selected results of this paper. We formulate general versions of these statements, as well as some other related statements, in Section \ref{mainres}. Note that part \ref{partAth1} of Theorem \ref{mainthm_uniform} and part \ref{part1meas} of Theorem \ref{mainthm_asympt} follow from arguments due to Jarn\'{\i}k, which we include for the sake of completeness; see their reformulations in Section \ref{mainres}.

\vskip+0.3cm

We begin with the uniform case. The sets $\widehat{ \bf D}[g]$, as well as their $\pmb{\eta}$-fixed sections $\{ \Theta: \,\,\, (\Theta, \pmb{\eta}) \in \widehat{ \bf D}[g]\}$, were actively studied before (see \cite{kwad} for a metric criterion for pairs, \cite{kimkim} for generalizations on Hausdorff measures, \cite{Agg25} for dimension bounds for sections). The recent paper \cite{Agg25} by Aggarwal also provides the upper bounds (Theorem 1.8) for the Hausdorff dimension of the sets $\widehat{ \bf D}^{\Theta} = \widehat{ \bf D}^{\Theta}[f_1]$ (more precisely, they work with bigger sets of trajectories divergent in average). Less was known about the sets $\widehat{ \bf D}^{\Theta}[g]$ for general $g$. Our first theorem describes some metric properties of these sets; in particular, we give answers to the questions stated in Section 7.3 in \cite{kwad} and Section 5.1 in \cite{kimkim}. For more general version of Theorem \ref{mainthm_uniform} \ref{partBmain}, see Theorem \ref{thm10}.
\vskip+0.3cm
Define 
\begin{equation}\label{kappadef}
    \kappa = 2^{1-m-n} ((m+n)!)^2.
\end{equation}

\begin{theorem}\label{mainthm_uniform}
    Fix $\Theta \in {\bf M}_{n,m}$. 
    \begin{enumerate}[label=(\alph*)]
        \item \label{partAth1} If $\Theta^{\top}$ is not $\kappa^m \!\cdot\! f$-approximable, then

        $$
        \widehat{\bf D}_{n,m}^{\Theta} \left[g\right] = \mathbb{R}^n;
        $$

        \item \label{partBmain} If  $\Theta^{\top}$ is $\left( \frac{\varepsilon}{m} \right)^m \!\cdot\!\left( f \circ \left( \frac{n}{\varepsilon} \right)^n \right)$-approximable for some $\varepsilon < \frac{1}{4}$, then the set

        $$
        \widehat{\bf D}_{n,m}^{\Theta} \left[g\right]
        $$

        has Lebesgue measure zero.
        
    \end{enumerate}
\end{theorem}

We note that Theorem \ref{mainthm_uniform} implies a recent result from \cite{kwad} (showing that the latter is equivalent to the Khintchine-Groshev's theorem); see Section \ref{kwadsection} for some discussion and a deduction.

\vskip+0.3cm

The second theorem answers some metrical questions about the sets $\widehat{ \bf W}_{n,m}^{\Theta} \left[g\right]$. The statement of Theorem \ref{mainthm_asympt}, part \ref{part2meas} can be thought of as a direct analog to the fact that "the set of badly approximable objects has zero measure" (cf. \cite{BDGW23}), where part \ref{part1meas} serves as a reference point to compare with; see Theorem \ref{nonSingtoBA_measure} for the language, corresponding to this description, and Theorem \ref{PDtoBA_HAW} for the winning properties of the corresponding sets. 

Part \ref{part3meas} utilizes the notation $ Y_\nu = |{\bf y}_\nu$| for the lengths of the best approximation vectors for $\Theta^{\top}$; we will introduce all the necessary definitions in Section \ref{bestappsection}. This statement provides a partial convergence condition for functions $g$ satisfying several technical restrictions; for the general version of this statement, see Theorem \ref{fullme}. We note that the optimal convergence-divergence condition for these sets is only known in case $n=m=1$ (see \cite{FK16}), and the proof relies on the fact that there are "no nontrivial singular numbers". 

\begin{theorem}\label{mainthm_asympt}
    Fix $\Theta \in {\bf M}_{n,m}$.
    \begin{enumerate}[label=(\alph*)]
        \item \label{part1meas} If $\Theta^{\top}$ is not $\kappa^m \!\cdot\! f$-Dirichlet, then
        $$
        \widehat{ \bf W}_{n,m}^{\Theta} \left[g\right] = \mathbb{R}^n.
        $$

        \item \label{part2meas} If there exists $\varepsilon > 0$, such that $\Theta^{\top}$ is not $\kappa^m \!\cdot\! \left( f \circ \left( \frac{1}{\varepsilon} \right) \right)$-Dirichlet, that is, 
        $$\Theta^{\top} \notin \bigcap\limits_{\varepsilon > 0} \widehat{\bf D}_{n,m} \left[\kappa^m \!\cdot\! \left( f \circ  \frac{1}{\varepsilon} \right) \right],
        $$
        then the set
        $$
        \widehat{ \bf W}_{n,m}^{\Theta} \left[g\right]
        $$

        has full Lebesgue measure in $\mathbb{R}^n$.

        \item \label{part3meas} Let $\lambda$ be a decreasing function on $\mathbb{R}_+$ such that (\ref{lambdatech}) holds. Let $g$ be such that conditions (\ref{technicalg}),  (\ref{additionalg}) and (\ref{finiteexpg}) hold. 

        Suppose 

            \begin{equation}\label{lambda_converges0}
   \sum\limits_{\nu = 1}^{\infty} \lambda^{\frac{1}{m}} \left( Y_{\nu}^n \right) < \infty,
      \end{equation}

      and let $\tilde{f}(T) = \lambda(T)^{\frac{n \gamma}{m}} f(T)$. If $\Theta^{\top}$ is $\tilde{f}$-Dirichlet, then the set 
      $$
        \widehat{ \bf W}_{n,m}^{\Theta} \left[g\right]
        $$

        has zero Lebesgue measure in $\mathbb{R}^n$.
      
    \end{enumerate}
\end{theorem}

\vskip+0.3cm

As mentioned above, the uniform setting is built upon a classical transference theorem by Jarn\'{\i}k (Theorem \ref{thmA}, see \cite{j39} for the original version),
which ensures existence of good inhomogeneous approximations for any $\pmb{\eta}$ under some conditions on $\Theta^{\top}$. This result is a natural analog of Dirichlet's theorem for this setup. We prove that the set of { \it Dirichlet improvable} $\pmb{\eta}$ has zero Lebesgue measure in this setup, and in addition -- that the set of non-singular $\pmb{\eta}$ is not only of full Lebesgue measure, but also winning. 

\vskip+0.3cm

The asymptotic (ordinary) setting also begins with the argument by Jarn\'{\i}k (Theorem \ref{nonSingtoBA_all}),
which provides an analog of Dirichlet's theorem asymptotic version. We show that that the generalization of bad approximability provides us with the set of measure zero, thus completing the result from \cite{K24}. We also show that this set of badly approximable $\pmb{\eta}$ is winning in the Hyperplane absolute game. Finally, we show a sufficient condition for a Khintchine-Groshev type theorem, showing that a set of shift vectors $\pmb{\eta}$ providing us with well enough approximable pairs $( \Theta, \eta)$ has measure zero.

\vskip+0.3cm

In addition, we revisit some classical results by Khintchine regarding the connection between ordinary homogeneous and uniform inhomogeneous approximation. In particular, Khintchine's approach gives an easy criterion for badly approximablity
(Corollary \ref{x1}).

\vskip+0.3cm

Our central technical tool is the classical transference principle. An  extensive theory of transference, that is, a connection between inhomogeneous approximations to  matrix $\Theta$ and  homogeneous  approximations ($\pmb{\eta}= \pmb{0}$) to the transposed  matrix $\Theta^\top$, was developed in classical works by 
A. Khintchine \cite{Kh,Kh1,Kh2,Kh3} and V. Jarn\'{\i}k \cite{j39,j41,j46}. A detailed exposition of the basics of this theory can be found in Chapter V of the book \cite{c} by J.W.S. Cassels.

\vskip+0.3cm
In the last  two decades various new  aspects  of the theory dealing with  metrical questions, winning properties, Hausdorff dimension, approximations on manifolds and various connections to Dynamical Systems on the spaces of lattices
were considered by many mathematicians. The majority of works discuss asymptotic inhomogeneous approximation  \cite{Beng17, BM17, BSV25, BHKV10, BFK, BFS, DS25, EinsTseng, FK16, Ha12, KL19, K24, kl_bad,  Mos11, Felipe, Shapi, Tseng}, however, uniform approximations were also actively studied, see \cite{Agg25, kimkim, four, kwad, N25}.

 \vskip+0.3cm

\subsection{Structure of the paper}

{\bf The structure of our paper} is as follows.
In {\bf Section \ref{obp}} we introduce all the necessary notions and notation. Here we should note that
some notation in recent papers 
differs from that in Cassels' book \cite{c} and earlier works.
  To make our exposition clearer we tried everywhere in the text to explain the meaning of modern results and their generalizations in terms of classical notation and compare the corresponding results.

\vskip+0.3cm

The main results of the paper are formulated in 
    {\bf Section \ref{mainres}}. For readers' convenience, we enumerate theorems which contain known results (sometimes slightly changed or reformulated) by letters A,B,C,... and theorems with new results by numbers 1,2,3,... \,.

 \vskip+0.3cm   

{\bf Section \ref{section31}} deals with uniform inhomogeneous approximation in connection with asymptotic approximation for the transposed matrix $\Theta^{\top}$.

{\bf Section \ref{section32}}
is devoted to asymptotic (or ordinary)
 inhomogeneous approximation and its connections with uniform approximation for $\Theta^{\top}$.

 {\bf Section \ref{subsection33}} covers the results arising from revision of Khintchine's ideas.

 { \bf Section \ref{examplessection}} shows some applications and corollaries of our results; specifically, in { \bf Section \ref{powerlog}} we show what our main theorems say when applied to functions of form $T^a \left( \log T \right)^b$. In { \bf Section \ref{badandvwa}} some old and recent characterizations of sets of badly and very well approximable matrices are discussed. In { \bf Section \ref{blsection}} we recall and slightly generalize the inequalities on Diophantine exponents due to Bugeaud and Laurent. In { \bf Section \ref{kwadsection}} we show how our results provide a simple proof of the zero-one law for pairs by D. Kleinbock and N. Wadleigh.

 \vskip+0.3cm

 Proofs of results from Section \ref{section31} and Section \ref{section32}, together with reformulations in the language of irrationality measure functions, will be given in { \bf Section \ref{alltransference}}. Specifically, we prove Theorem \ref{AtoSING_HAW}, Theorem \ref{PDtoBA_HAW}, Theorem \ref{fullme}, and in addition Theorem \ref{thmA} and Theorem \ref{nonSingtoBA_all} in { \bf Section \ref{transferencesection}}; we prove Theorem \ref{thm10} in {\bf Section \ref{thm10proofsect}} and Theorem \ref{nonSingtoBA_measure} in {\bf Section \ref{BAnullsection}}.

 \vskip+0.3cm

 {\bf Section \ref{kkh}} presents the proof of Theorem \ref{Khintchine_general}, involving Khintchine's original construction.

 \vskip+0.3cm

 The remaining sections are technical. In {\bf Section \ref{bestapproximations}} we introduce the language of best approximation vectors and irrationality measure functions which will be used for proofs, and which allows one to compare modern and classical results. Specifically, in { \bf Section \ref{bestappsection}} we prove a simple property of the irrationality measure function that may be of separate interest, and { \bf Section \ref{asympd}} provides a short discussion about asymptotic directions of the best approximations vectors, which will be important for some theorems in Section \ref{mainres}.

 {\bf Section \ref{measureanddim}} collects some technical statements which allow us to generalize our results to certain affine subspaces: we introduce the sets of exceptions to our theorems in {\bf Section \ref{exceptional}} and prove some technical metrical lemmata in {\bf Section \ref{measuretechnical}}.

 {\bf Section \ref{deductionsection}} shows how Theorem \ref{measure_corollary} and Theorem \ref{mainthm_asympt} \ref{part3meas} are deduced from Theorem \ref{fullme}.

\section*{Acknowledgements}

The authors are grateful to Dmitry Kleinbock for many helpful comments and suggestions. They also thank Pablo Shmerkin and Mike Hochman for referencing useful facts regarding Hausdorff dimensions.

\vskip+0.3cm
 Nikolay Moshchevitin's research is  supported by Austrian Science Fund (FWF),
 Forschungsprojekt PAT1961524.
 

\section{Definitions and conventions}\label{obp}

In Section \ref{definitionssec} we introduce all the necessary notions needed to formulate our results. In Section \ref{technicalsect} we list additional technical conditions on approximating functions $f$ and $g$ which will be used in some of the statements.

\subsection{Definitions}\label{definitionssec}

\subsubsection{Bad and very good approximability}
The definition of $g$-approximable matrices given above motivates a natural generalization of the notion of badly approximable pairs. The classical definition (see \cite{kl_bad} for a precise formulation) says that a pair is badly approximable if, up to a constant, this pair can not be asymptotically approximated by a function decreasing faster than $f_1$. One can replace $f_1$ by an arbitrary function $g$ and obtain 

\begin{definition}\label{baddef}
    The set $\widehat{{ \bf BA}}_{n, m}\left[ g \right]$ of {\it inhomogeneously $g$-badly approximable} pairs $(\Theta, \pmb{\eta})$ is the set of pairs for which 
$$
\liminf\limits_{|{\bf q}| \rightarrow \infty} \frac{||\Theta {\bf q} - \pmb{\eta}||_{\mathbb{Z}^n}^n}{g(|{\bf q}|^m)} > 0.
$$
The sets
$$
\widehat{{ \bf BA}}_{n, m}^{\Theta} \left[ g \right] = \{ \pmb{\eta} \in \mathbb{R}^n: \,\,\,  (\Theta, \pmb{\eta}) \in \widehat{{ \bf BA}}_{n, m}\left[ g \right] \}
$$
will be called {\it sections}
of
$\widehat{{ \bf BA}}_{n, m}[ g]$.

\end{definition}

One can also naturally define the set of $f$-badly approximable matrices as

$$
{{ \bf BA}}_{n, m} \left[ f \right]: = \{ \Theta \in {\bf M}_{n,m}: \,\,\, (\Theta, {\bf 0 }) \in \widehat{{ \bf BA}}_{n, m}\left[ g \right] \}.
$$

We will prove that under certain choices of $g$ the set $\widehat{{ \bf BA}}_{n, m}^{\Theta} \left[ g \right]$ is a Lebesgue null set (Section \ref{badisnull}), but its restrictions on certain subspaces are winning in variations of Schmidt's game (Section \ref{badiswinning}). We will also discuss some necessary and sufficient conditions for a matrix to be badly approximable in the classical sense (that is, belong to the set ${{ \bf BA}}_{n, m} \left[ f_1 \right]$) in Section \ref{badandvwa} (where $f_1$  is defined in (\ref{qq1})).

\vskip+0.3cm

Another notion we will briefly discuss in Section \ref{badandvwa} is very good approximability, which we will only formulate for matrices:

\begin{definition}
    A matrix $\Theta \in {\bf M}_{n,m}$ is very well approximable if there exists $\varepsilon > 0$, such that
$$
\Theta \in {\bf W}_{n,m}[f_{1 + \varepsilon}].
$$
\end{definition}

\subsubsection{Singularity}

Analogously to the classical homogeneous case, one can naturally define the notion of singularity:

\begin{definition}\label{singulardef}
We say that the pair $(\Theta, \pmb{\eta})$ is 
{ \it $g$-singular} if 
$$
(\Theta, \pmb{\eta}) \in \bigcap\limits_{\varepsilon > 0} \widehat{{\bf D}}_{n,m} \left[ \varepsilon g \right].
$$
We denote the set of $g$-singular pairs by $\widehat{ \bf Sing}_{n,m} \left[ g \right]$, and

$$
\widehat{ \bf Sing}_{n,m}^{\Theta} \left[ g \right] = \Big\{  \pmb{\eta} \in \mathbb{R}^n: \,\,\, (\Theta, \pmb{\eta}) \in \widehat{ \bf Sing}_{n,m}\left[ g \right] \Big\}.
$$

\end{definition}

We will call a pair $(\Theta, \pmb{\eta})$  simply {\it singular} if it is $f_1$-singular
with $f_1$ defined in (\ref{qq1}) and denote the set of all  singular pairs $\widehat{ \bf Sing}_{n,m}$.

\vskip 0.3 cm

We call a matrix $\Theta \in { \bf M}_{n,m}$ {\it trivially singular} if there exists such a ${\bf q}\in \mathbb{Z}^m$ for which $\Theta {\bf q} \in \mathbb{Z}^n$. We call a pair $(\Theta, \pmb{\eta})$ {\it trivially singular} if there exists such a ${\bf q}\in \mathbb{Z}^m$ for which $\Theta {\bf q} - \pmb{ \eta} \in \mathbb{Z}^n$. It is easy to see that a pair $(\Theta,\pmb{\eta})$ is trivially singular if and only if the  vector $\pmb{\eta}$ can be written as a linear combination of the vectors 
\begin{equation}\label{vect}
\left( \begin{array}{c} \theta_{1,1}\cr \vdots \cr \theta_{n,1}\end{array} \right) ,
\dots ,
\left( \begin{array}{c} \theta_{1,m}\cr \vdots \cr \theta_{n,m}\end{array} \right),
\left( \begin{array}{c} 1\cr \vdots \cr 0\end{array} \right)
,\dots,
\left( \begin{array}{c} 0\cr \vdots \cr 1\end{array} \right)
\end{equation}
with integer coefficients.

\subsubsection{Diophantine exponents}

Now let us define Diophantine exponents involved in our consideration. The {\it ordinary Diophantine exponent} $\omega(\Theta, \pmb{\eta})$ of a pair $(\Theta, \pmb{\eta})$ is defined as the supremum of positive reals $\gamma$, for which the inequality 
$$
||\Theta {\bf q} - \pmb{\eta}||_{\mathbb{Z}^n} \leq |{\bf q}|^{-\gamma}
$$
has infinitely many solutions ${\bf q} \in \mathbb{Z}^m \setminus \{ {\bf 0} \}$. The {\it uniform Diophantine exponent} $\hat{\omega}(\Theta, \pmb{\eta})$ of a pair $(\Theta, \pmb{\eta})$ is defined as the supremum of positive reals $\gamma$, for which the system 
$$
\begin{cases}
    ||\Theta {\bf q} - \pmb{\eta}||_{\mathbb{Z}^n} \leq t^{-\gamma} \\
    |{\bf q}| \leq t
\end{cases}
$$
has solutions ${\bf q} \in \mathbb{Z}^m \setminus \{ {\bf 0} \}$ for any $t \in \mathbb{R}_+$ large enough. We define the corresponding ordinary and uniform Diophantine exponents of a matrix $\Theta$ by 

$$
\omega(\Theta) = \omega(\Theta, {\bf 0}) \,\,\,\,\,\,\,\,\,\,\,\,\text{and} \,\,\,\,\,\,\,\,\,\,\,\, \hat{\omega}(\Theta) = \hat{\omega}(\Theta, {\bf 0}).
$$

\subsubsection{HAW property}

Lastly, we briefly discuss the Hyperplane absolute winning game and the HAW property.

Hyperplane absolute game is a version of the classical two-player Schmidt's game (see \cite{S66} for the Schmidt's $(\alpha, \beta)$-game setup) introduced by Broderick, Fishman, Kleinbock, Reich and Weiss in \cite{HAW}, which provides the notion of {\it winning} stronger than the original Schmidt's game does. We call the set {\it Hyperplane absolute winning (HAW)} if it is winning in this game. We will not need the definitions and rules of the game and thus will not define it here, just mentioning several simple properties of HAW sets which are of interest for this paper:

\begin{itemize}
    \item If a set is HAW, it is also winning in the classical Schmidt's game. More precisely, HAW implies $\alpha$-winning for all $\alpha < \frac{1}{2}$.
    \item A countable intersection of HAW sets is HAW.
    \item HAW sets have full Hausdorff dimension.
\end{itemize}

There are other important properties of HAW sets (for instance, HAW property is inherited by diffuse sets, which guarantees that HAW sets are winning on certain fractals). To keep this paper short, we do not introduce additional required definitions here; the reader can find the necessary definitions and proofs in \cite{HAW}.

\subsection{Technical conditions on functions $f$ and $g$}\label{technicalsect}

In this subsection, we collect some technical restrictions on functions $f$ and $g$ we will refer to throughout the text. None of the restrictions are needed in the general form of our results, however they make some resuls look much easier and clearer.

The first and most widely used condition is the following: 
\begin{equation}\label{technical}
  \textit{there exists } \,\, \gamma > 0 \,\, \textit{ that for any} \,\, T \in \mathbb{R}_+ \,\,\textit{ and any } \,\, C>1 \,\, 
    \textit{large enough, one has} \,\, f(CT) \leq C^{- \gamma} f(T).
\end{equation}

Property (\ref{technical}) is natural: it is equivalent to the existence of such $K > 1$ that 
$$\liminf\limits_{T \rightarrow \infty} \frac{f(T)}{f(KT)} > 1.$$

For instance, it is satisfied by all the functions we use in Section \ref{examplessection}. For $f$ and $g$ related via (\ref{relationg(f)}), this is equivalent to saying that 

\begin{equation}\label{technicalg}
    g(CT) \geq C^{-\frac{1}{\gamma}} g(T)\,\,\,\, \text{for any} \,\, T > 1 \,\, \text{and for any} \,\, C \,\, \text{large enough}.
\end{equation}

\vskip+0.3cm

The remaining stricter conditions are only used in Theorem \ref{measure_corollary} and Theorem \ref{mainthm_asympt} \ref{part3meas}.

The condition on $\lambda$ needed in both of the theorems is:

\begin{equation} \label{lambdatech}
        \text{There exists} \,\, \alpha > 0, \,\, \text{such that} \,\,\lambda(T^{\delta}) \geq \delta^{- \alpha} \lambda(T) \,\, \text{for any} \,\, \delta > 1, \,\, T > 0.
        \end{equation}

\vskip+0.3cm

The conditions on $f$ for Theorem \ref{measure_corollary} are as follows: 

\begin{equation}\label{additionalf}
            \frac{f(T_2)}{f(T_1)} < \left( \frac{\log T_1}{\log T_2} \right)^\alpha \,\,\,\,\,\,\,\,\,\, \text{for any} \,\, T_2 > T_1,
        \end{equation}

        where $\alpha$ is as in (\ref{lambdatech}) (which says that $f$ locally decreases not too slowly), and 

        \begin{equation}\label{finiteexp}
            f(T) \geq T^{-\beta} \,\,\,\,\,\,\,\,\ \text{for some} \,\, \beta > 0
        \end{equation}

        (which means that the uniform exponent $\hat{\omega}(\Theta^\top)$ is finite).

\vskip+0.3cm

Finally, below are similar conditions needed in Theorem \ref{mainthm_asympt} \ref{part3meas}.

\begin{equation}\label{additionalg}
            \frac{g(T_2)}{g(T_1)} < \left( \frac{\log T_1}{\log T_2} \right)^{\frac{n}{m}\alpha} \,\,\,\,\,\,\,\,\,\, \text{for any} \,\, T_2 > T_1,
        \end{equation}

        where $\alpha$ is as in (\ref{lambdatech}) and

        \begin{equation}\label{finiteexpg}
            g(T) \leq T^{-\frac{1}{\beta}} \,\,\,\,\,\,\,\,\ \text{for some} \,\, \beta > 0.
        \end{equation}

\vskip+0.3cm

\section{Main results}\label{mainres}



Before stating our main results, we would like to mention that the case of trivially singular $\Theta^{\top}$ provides us with "the worst possible" approximation properties for almost all the pairs $( \Theta, \pmb{\eta})$: namely, the form $|| \Theta {\bf q} - \pmb{\eta}||$ is uniformly bounded from below for all ${\bf q} \in \mathbb{Z}^m$. We will formulate a more precise statement in Remark \ref{trivsingremark}.

\vskip+0.3cm

\subsection{The duality between asymptotic homogeneous approximations of a transposed matrix and uniform inhomogeneous approximations}\label{section31}

{ In this section we discuss relations between  asymptotic approximations of matrix $\Theta^{\top}$ and uniform approximations of pairs $(\Theta, \pmb{\eta})$; that is, we prove statements of the form:

{
\it Suppose we know that $\Theta^{\top}$ belongs/does not belong to ${\bf W}_{m,n}[f]$ for certain functions $f$. What can we say about the sets of form $\widehat{ \bf D}_{n,m}^{\Theta}[g]$?
}
}

\subsubsection{Dirichlet's Theorem (due to Jarn\'{\i}k)}

We start this section with the revision of classical results. Theorem \ref{thmA} (which is a reformulation of Theorem \ref{mainthm_uniform} \ref{partAth1}) follows from the argument by Jarn\'{\i}k from a seminal paper
\cite{j39} (see also \cite{j41,j46}). This Theorem can be thought of as an analog of the classical Dirichlet's theorem for a fixed $\Theta$. 

\vskip+0.3cm
\noindent
\begin{thm}\label{thmA}
  Suppose the transposed matrix $\Theta^\top$ is not $f$-approximable, that is,  
  $$
  \Theta^\top \notin { \bf W}_{m,n}\left[f\right].
  $$
  
  Then, any pair $(\Theta, \pmb{\eta})$ is $g \circ \frac{1}{\kappa^m}$-Dirichlet (where $\kappa$ is defined in \eqref{kappadef}), that is, 

  $$
  \widehat{\bf D}^{\Theta} \left[ g \circ \frac{1}{\kappa^m} \right] = \mathbb{R}^n.
  $$
  
  \vskip+0.2cm

  If we assume in addition that $f$ satisfies condition (\ref{technical}), then 
  $$
  \widehat{\bf D}^{\Theta} \left[ Kg  \right] = \mathbb{R}^n
  $$
  for some $K$ large enough.
\end{thm}

\vskip+0.3cm

\subsubsection{Davenport-Schmidt's Theorem}

The classical theorem of Davenport and Schmidt (see \cite{DS1, DS2}) states that for almost all matrices $\Theta$ Dirichlet's theorem can not be improved:

$$
\text{If } \varepsilon < 1,\text{ then the Lebesgue measure of the set } {\bf D}_{n,m}[ \varepsilon f_1] \text{ is zero}.
$$

We will prove a natural analog of this statement in the inhomogeneous fixed matrix case: up to a constant (multiplicative and, potentially, inside the argument of function $g$), Theorem \ref{thmA} can not be improved.



More specifically, it will be shown that this statement holds in restriction on certain affine subspace $\mathcal{A}$.
We note that Theorem \ref{mainthm_uniform} \ref{partBmain} corresponds to the case $\mathcal{A}=\mathbb{R}^n$.
To formulate this result we need a notion of
{ \it $\Xi$-exceptional} affine subspaces (here $\Xi$ is some positive real parameter). The exact definition is a little bit cumbersome and is related to asymptotic directions for the best approximation vectors. It will be given in Section \ref{exceptional}.
In fact, in Section  \ref{exceptional} we show the the set of all
non
$\Xi$-exceptional affine subspace
is a set of full measure and give an upper bound for the Hausdorff dimension of the set of 
$\Xi$-exceptional subspaces in Corollary \ref{dimensionlemma}.

\begin{theorem}\label{thm10}
Suppose $\Theta^{\top}$ is an $f$-approximable matrix. 

 Let $\mathcal{A}$ be a non-$\Xi$-exceptional affine subspace of $\mathbb{R}^n$  for some $0 < \Xi < 1$. 
    If $\varepsilon < \frac{1}{4}$, then the set 
    $$
    \widehat{\bf D}_{n,m}^\Theta \left[ \left( \frac{\varepsilon}{n} \right)^n \cdot g \circ \left( \frac{m}{\varepsilon} \right)^m \right] \cap \mathcal{A}
        $$
 has Lebesgue measure zero in $\mathcal{A}$.

 

  If we assume in addition that $f$ satisfies condition (\ref{technical}), then the set
  $$
  \widehat{\bf D}^{\Theta} \left[ \frac{1}{K}g  \right]
  $$
  has Lebesgue measure zero in $\mathcal{A}$ for some $K$ large enough.

\end{theorem}

\subsubsection{The set of singular vectors is winning}

Theorem \ref{thm10} shows in particular that the set of non-singular shift vectors $\pmb{\eta}$ has full Lebesgue measure. Assuming condition (\ref{technical}), we can strengthen this statement, showing that this set of full measure is also winning in the Hyperplane Absolute Winning game. We prove this winning property in a more general form, in restriction on certain affine subspaces, however now the condition on subspaces is different and involves the set of { \it asymptotic directions} of the matrix $\Theta^{\top}$. We introduce the asymptotic directions set and briefly discuss its properties in Section \ref{asympd}.

\begin{theorem}\label{AtoSING_HAW}
     Suppose $f$ satisfies condition (\ref{technical}), and $\Theta^{\top}$ is an $f$-approximable matrix. Let $\Omega$ be the set of all the asymptotic directions for the best approximations of $\Theta^{\top}$.
    
        Let $\mathcal{A}$ be an affine subspace of positive dimension, with $\mathcal{L}$ as the corresponding linear subspace and $\mathcal{L}^{\perp}$ its orthogonal complement.  If $\Theta^{\top}$ is trivially singular, assume also that $\mathcal{A}$ is not exceptional.

        If $\mathcal{L}^{\perp} \cap \Omega = \emptyset$, then the set $$
        \left( \widehat{ \bf Sing}_{n,m}^{\Theta} \left[ g \right] \right)^c \cap \mathcal{A}
        $$
        is HAW in $\mathcal{A}$. In particular, the set 
        $$
        \left( \widehat{ \bf Sing}_{n,m}^{\Theta} \left[ g \right]\right)^c
        $$
is HAW.
\end{theorem}

A more general version of Theorem \ref{AtoSING_HAW} (without assuming \eqref{technical}) is formulated in terms of irrationality measure functions in Theorem \ref{Dyakova}.

\vskip+0.3cm

We note that $\Theta \in {\bf BA}_{n,m}$ is badly approximable\footnote{That is, belongs to the set ${{ \bf BA}}_{n, m} \left[ f_1 \right]$} if and only if $\Theta^\top \in {\bf BA}_{m,n}$ is badly approximable. In this case, Corollary \ref{x1} gives an even stronger statement: the set 
$$
        \left( \widehat{ \bf Sing}_{n,m}^{\Theta} \left[ g \right]\right)^c
$$

is countable.

\subsection{The duality between uniform homogeneous approximations of a transposed matrix and asymptotic inhomogeneous approximations}\label{section32}

{ In this section we discuss relations between uniform approximations of matrix $\Theta^{\top}$ and asymptotic approximations of pairs $(\Theta, \pmb{\eta})$. That is, we prove statements of a form:

{
\it Suppose we know that $\Theta^{\top}$ belongs/does not belong to ${\bf D}_{m,n}[f]$ for certain functions $f$. What can we say about the sets of form $\widehat{ \bf W}_{n,m}^{\Theta}[g]$?
}
}

\subsubsection{Asymptotic Dirichlet's Theorem (due to Jarn\'{\i}k)}

As in Section \ref{section31}, we start with a general result which follows from Jarn\'{\i}k's arguments. This result can be treated as analogous to the classical asymptotic corollary from Dirichlet's theorem in the inhomogeneous case with a fixed matrix $\Theta$. It is a general version of the statement of Theorem \ref{mainthm_asympt} \ref{part1meas}.

\begin{thm}\label{nonSingtoBA_all}
Suppose $\Theta^{\top}$ is a non $f$-Dirichlet matrix. 
Then, for any $\pmb{\eta} \in \mathbb{R}^n$ the pair $(\Theta, \pmb{\eta})$ is $g \left( \frac{1}{\kappa^m} T \right)$-approximable (with $\kappa$ as in \eqref{kappadef}):

        $$
        \widehat{{ \bf W}}^{\Theta}_{n, m}\left[ g \circ \frac{1}{\kappa^m} \right] = \mathbb{R}^n.
        $$

  If we assume in addition that $f$ satisfies the condition (\ref{technical}), it implies that 
  $$
  \widehat{\bf W}^{\Theta} \left[ Kg  \right] = \mathbb{R}^n
  $$
  for some $K$ large enough.

\end{thm}

\vskip+0.3cm

\subsubsection{BA is null}\label{badisnull}


The next theorem shows: if not all vectors $\pmb{\eta}$ belong to $\widehat{\bf W}_{n,m}^{\Theta}[g]$, then the set of $g$-badly approximable shifts $\pmb{\eta}$ has zero Lebesgue measure. We note that in a classical special case $g = f_1$ of badly approximable pairs this was shown for non-singular\footnote{Non-$f_1$ singular in our notation} $\Theta$ in \cite{K24} (Corollary 1.4).

\begin{theorem}\label{nonSingtoBA_measure} 
    Suppose $\Theta^{\top}$ is a non $f$-Dirichlet matrix. Then, the set 

        $$
        \widehat{{ \bf BA}}_{n, m}^{\Theta} \left[ g \circ \frac{1}{\kappa^m}  \right]
        $$

         has Lebesgue measure zero in $\mathbb{R}^n$. 
         If we assume that condition (\ref{technical}) holds, then the set 
        $$
        \widehat{{ \bf BA}}_{n, m}^{\Theta} \left[ g \right]
        $$

        of  $g$-badly approximable shift vectors $\pmb{\eta}$ has measure zero.
    
\end{theorem} 

Note that part \ref{part2meas} of Theorem \ref{mainthm_asympt} is a reformulation of Theorem \ref{nonSingtoBA_measure}.

\subsubsection{BA is winning}\label{badiswinning}

Despite being of measure zero, the set $\widehat{{ \bf BA}}_{n, m}^{\Theta} \left[ g \right]$ is known to have full Hausdorff dimension, and moreover, has a stronger winning property. The dimension and winning properties of these sets were widely studied before. It was shown by Kleinbock in \cite{kl_bad} that the set $\widehat{{ \bf BA}}_{n, m}[f_1]$ of badly approximable pairs has full Hausdorff dimension, and by Bugeaud, Harrap, Kristensen and Velani in \cite{BHKV10} that the set $\widehat{{ \bf BA}}_{n, m}^{\Theta} \left[ f_1 \right]$ has full Hausdorff dimension on certain fractals. It was shown by Tseng \cite{Tseng} that the set $\widehat{{ \bf BA}}_{1, 1}^{\Theta} \left[ f_1 \right]$ is winning in Schmidt's $(\alpha, \beta)$-game. It was then shown by Moshchevitin (\cite{Mos11}, Theorem 2) that the set $\widehat{{ \bf BA}}_{n, m}^{\Theta} \left[ g \right]$ is winning, provided that the matrix $\Theta^{\top}$ is $f$-Dirichlet and $g(T) = \frac{1}{f^{-1} \left( \frac{1}{T} \right)}$. In particular, it implies that $\widehat{{ \bf BA}}_{n, m}^{\Theta} \left[ f_1 \right]$ is winning. It was shown by Einsiedler and Tseng in \cite{EinsTseng} (Theorem 1.4) and also independently by Broderick, Fishman and Kleinbock in \cite{BFK} (Corollary 1.4) that the sets $\widehat{{ \bf BA}}_{n, m}^{\Theta} \left[ f_1 \right]$ are winning on certain fractals, and strengthened by Broderick, Fishman and Simmons in \cite{BFS} (Theorem 1.1) who proved that the aforementioned set is HAW. We use similar techniques to prove our Theorem \ref{PDtoBA_HAW} which generalizes all these results, and proves the winning properties on certain subspaces. We note that a more general result (without the restriction \eqref{technical}) is formulated in terms of irrationality measure functions, and can be found in Theorem \ref{Dyakova}.

\begin{theorem}\label{PDtoBA_HAW}
    Suppose $f$ satisfies condition (\ref{technical}), and $\Theta^{\top} \in {\bf D}_{m,n} \left[ f \right]$ is an $f$-Dirichlet matrix. Let $\Omega$ be the set of all the asymptotic directions for the best approximations of $\Theta^{\top}$.
    
        Let $\mathcal{A}$ be an affine subspace of positive dimension, with $\mathcal{L}$ as the corresponding linear subspace and $\mathcal{L}^{\perp}$ its orthogonal complement.  If $\Theta^{\top}$ is trivially singular, assume also that $\mathcal{A}$ is not exceptional.

        If $\mathcal{L}^{\perp} \cap \Omega = \emptyset$, then the set $ \widehat{{ \bf BA}}_{n, m}^{\Theta} \left[ g \right]  \cap \mathcal{A}$ is HAW in $\mathcal{A}$. In particular, the set 
        $$
        \widehat{{ \bf BA}}_{n, m}^{\Theta} \left[ g \right] 
        $$

        is HAW.

\end{theorem}

Let us notice that winning properties of the sets $\widehat{{ \bf BA}}_{n, m}^{\Theta} \, \left( = \widehat{{ \bf BA}}_{n, m}^{\Theta}\left[ f_1 \right] \right)$ on affine subspaces were studied for the case of badly approximable $\Theta$ in \cite{Beng17}: they proved (Theorem 2.4), in the weighted setting, that the aforementioned sets are winning in the classical Schmidt's $(\alpha, \beta)$-game on any affine subspace $\mathcal{A}$. One can ask if the winning property on arbitrary affine subspaces holds for the general $\Theta$ (and thus, if our restrictions on $\mathcal{A}$ in Theorem \ref{PDtoBA_HAW} can be omitted). The negative answer to this question is given in \cite{D19}; namely, they prove the following statement (Theorem 3.1):

\begin{proposition}\label{Dyakova_21_result}
   There exists such a vector ($2 \times 1$-matrix) $\Theta = \begin{pmatrix}
       \theta_1 \\
       \theta_2
   \end{pmatrix}$ and such an affine line $\mathcal{A} \subset \mathbb{R}^2$ that 
   \begin{enumerate}[label=(\alph*)]
       \item The numbers $1, \theta_1, \theta_2$ are linearly independent over $\mathbb{Q}$; 
       \item The set 
       $$ \widehat{{ \bf BA}}_{2, 1}^{\Theta} \left[ f_1 \right]  \cap \mathcal{A}$$
   \end{enumerate}
   is not winning in $\mathcal{A}$ (in the classical Schmidt's game), and in particular is not HAW.
\end{proposition}

        \vskip 0.3 cm 

\subsubsection{Kurzweil-type theorems}

In \cite{KU}, Kurzweil proved a famous result:

\begin{thm}\label{kurzweil}
    Suppose $\Theta \in {\bf BA}_{n,m}[f_1]$, that is, $\Theta$ is a badly approximable matrix. The Lebesgue measure of the set $\widehat{\bf W}_{n,m}^{\Theta}[g]$ is zero if the series
    $$
    \sum\limits_{k=1}^{\infty} k^{m-1} g^n(k)
    $$

converges, and full if it diverges. 
\end{thm}

Remarkably, this convergence conditions is shown to hold if and only if the matrix $\Theta$ is badly approximable. This raises a natural general question: given a fixed $\Theta \in {\bf M}_{n,m}$, what are the conditions on function $g$ for the set $\widehat{\bf W}_{n,m}^{\Theta}[g]$ being a nullset or a full measure set?

The necessary and sufficient condition (which is also a convergence-divergence condition for a certain series) is only known when $\Theta \in \mathbb{R}$; see \cite{FK16}. Theorem \ref{nonSingtoBA_measure} provides a condition under which a set $\widehat{\bf W}_{n,m}^{\Theta}[g]$ has full Lebesgue measure. In this section we want to show a (non-optimal) condition under which $m \left(\widehat{\bf W}_{n,m}^{\Theta}[g] \right)=0$.


\vskip+0.3cm

We will need some additional notation. First, our result will use the sequence $\{ {\bf y_{\nu}} \}$ of the best approximation vectors for $\Theta^{\top}$ and the sequence of their lengths $Y_\nu = |{\bf y_{\nu}}|$. The definition and discussion of these objects can be found in Section \ref{bestapproximations}.

Next, let us consider a function $\lambda(t)$, satisfying two conditions:

    \begin{equation}\label{lambda_converges}
   \sum\limits_{\nu = 1}^{\infty} \lambda^{\frac{1}{m}} \left( Y_{\nu}^n \right) < \infty,
      \end{equation}
      and  
    \begin{equation}\label{monot_cond}
        \text{The function $\frac{\lambda(T)}{f(T)}$ monotonically increases.}
    \end{equation}

Let $H$ be the inverse function to the function $\frac{\lambda}{f}$. Finally, let
\begin{equation}\label{lltildeg}
\tilde{g}(T) = \tilde{g}_{n,m}(T) = T^{\frac{n}{m}} \cdot \frac{f^{\frac{n}{m}} \left( H(T) \right)}{H(T)} = \frac{\lambda^{\frac{n}{m}} \left( 
H(T) \right)}{H(T)}.
\end{equation}

Then, the following statement holds:

\begin{theorem}\label{fullme}
    Suppose that $\Theta^{\top}$ is an $f$-Dirichlet matrix.
    Let $\mathcal{A}$ be a non-1-exceptional affine subspace of $\mathbb{R}^n$.

        Then, for any $\varepsilon > 0$ and almost any $\pmb{\eta} \in \mathcal{A}$ the pair $(\Theta, \pmb{\eta})$ is $\frac{1 - \varepsilon}{2n (2m)^n}\!\cdot\!\tilde{g}$-Badly approximable, or, equivalently, the set 

        $$
        \widehat{\bf W}^{\Theta}_{n,m} \left[ \frac{1 - \varepsilon}{2n(2m)^n } \tilde{g} \right] \cap \mathcal{A}
        $$
        
        has zero Lebesgue measure in $\mathcal{A}$.
\end{theorem}

Theorem \ref{fullme} follows from Theorem \ref{D_measure}; see Section \ref{transferencesection} for details.

We can simplify the statement of Theorem \ref{fullme} if we assume that some additional technical conditions hold for the functions $\lambda$ and $f$.

\begin{theorem}\label{measure_corollary}
    
    Suppose that $\Theta^{\top}$ is an $f$-Dirichlet matrix, $\lambda$ is a decreasing function on $\mathbb{R}_+$, and conditions (\ref{technical}),  (\ref{lambdatech}),  (\ref{additionalf}) and (\ref{finiteexp}) are satisfied.

    Let $u(T) = \lambda \left( T \right)^{\frac{m + \gamma n}{m \gamma}} g(T)$.
        
        Let $\mathcal{A}$ be a non-1-exceptional affine subspace of $\mathbb{R}^n$.

If
        \begin{equation}\label{lambdaseries}
        \sum\limits_{\nu = 1}^{\infty} \left( \frac{u(Y_{\nu}^n)}{g(Y_{\nu}^n)} \right)^{\frac{\gamma}{m + \gamma n}} =  \sum\limits_{\nu = 1}^{\infty} \lambda^{\frac{1}{m}} \left( Y_{\nu}^n \right) < \infty,
        \end{equation}

    then the set 
    $$
    \widehat{\bf W}_{n,m}^{\Theta} [u] \cap \mathcal{A}
    $$
    of $u$-approximable shift vectors $\pmb{\eta}$ in $\mathcal{A}$ has Lebesgue measure 0 in $\mathcal{A}$.
\end{theorem}

Let us note that we can always find such a $u$ that satisfies conditions (\ref{lambdatech}) and (\ref{lambdaseries}): one can always take 
$$
\lambda(T) = \frac{1}{\log(T)^{m(1 + \varepsilon)}}
$$

for $\varepsilon > 0$. The convergence in (\ref{lambdaseries}) is guaranteed by Proposition \ref{exp} \ref{prop4A}.

\subsection{Khintchine's results revisited}\label{subsection33}

We begin this section with formulation of a classical result by Khintchine  \cite{Kh1,Kh2}.
The following is just a reformulation of the original result.

\begin{thm}\label{x2}
For any pair $(\Theta, \pmb{\eta})$ that is not trivially singular one has $ \omega(\Theta) \ge \hat{\omega}(\Theta, \pmb{\eta})$,
meanwhile for any $\Theta$ there exists $\pmb{\eta}$ with $\hat{\omega}(\Theta, \pmb{\eta})=\omega(\Theta) $.
In other words
$$
\omega(\Theta) = \sup  \hat{\omega}(\Theta, \pmb{\eta}),
$$
where the supremum is taken over such $\pmb{\eta}\in \mathbb{R}^n$ that the pair $(\Theta, \pmb{\eta})$ is not trivially singular.

\end{thm}

 A more general result looks as follows.

\begin{theorem}\label{Khintchine_general}
Fix a non trivially singular $\Theta \in {\bf M}_{n,m}$ and
$\varepsilon > 0$. Let $h(T)$ be a non-increasing function.

  \vskip+0.3cm

\begin{enumerate}[label=(\alph*)]
    \item \label{Kh_partA} If $\Theta$ is not $(1 + \varepsilon) 2^n \cdot h \circ \frac{1}{2^m}$-approximable, then

   $$
   \widehat{\bf D}_{n,m}^{\Theta} \left[ h \right] = \emptyset.
   $$

   \item \label{Kh_partB} Let
        \begin{equation}\label{ab}
 A = 3^{n+m}-1\,\,\,\text{and} \,\,\, B = 2^{m-1}\cdot (2^{2n} - 2^{n}).
 \end{equation}

If $\Theta$ is $\frac{(1 - \varepsilon)}{2^n B^n} \cdot h \circ  \left( 2^m A^m \right)$-approximable, then the set

   $$
   \widehat{\bf D}_{n,m}^{\Theta} \left[ h \right] 
   $$
is uncountable.
\end{enumerate}

\end{theorem}

\vskip+0.3cm

In Section \ref{kkh} we formulate and prove  a general result in terms of irrationality measure functions (Theorem \ref{main}) which leads to Theorem \ref{Khintchine_general}. Here we would like to make some comments on Theorem \ref{Khintchine_general}.

 \vskip+0.3cm
First of all we would like to formulate  a very simple  criterion of bad approximability which may be compared with the results of recent paper \cite{Felipe} related to a theorem by Kurzweil \cite{KU}.

\begin{corollary}\label{x1}
    \it $\Theta$ is a badly approximable matrix if and only if there is no vector $\pmb{\eta} \in \mathbb{R}^n$ such that the pair $(\Theta, \pmb{\eta})$ is singular (and even $\varepsilon f_1$-approximable for small enough $\varepsilon$), except for vectors of  form $\Theta \mathbb{Z} + \mathbb{Z}$ (for which the pair $(\Theta, \pmb{\eta})$ is just trivially singular).
\end{corollary}

Also as a corollary of Theorem \ref{Khintchine_general} we can deduce a general version of a result 
announced in \cite{Mo}
  (alternative approaches and general related results are discussed in  \cite{four}, Section 8.2).

  \begin{corollary}\label{x1x}
 For any function $h (T)$ decreasing to zero there exists an $m\times n$ matrix $\Theta$
and a vector $\pmb{\eta} \in \mathbb{R}^n$, such that the pair $(\Theta, \pmb{\eta})$ is $h$-Dirichlet. 
\end{corollary}

\vskip + 0.3 cm

It is natural to ask whether it is possible to generalize the statement \ref{Kh_partB} of Theorem \ref{Khintchine_general} to affine subspaces of $\mathbb{R}^n$. Here we want to explain why this is not always possible in a simple low-dimensional case. One can easily construct straightforward higher dimensional generalizations of this statement.

\begin{remark}
   Let $h(T)$ be a non-increasing function, with an additional condition that $h(T) = o (f_2(T))$. 
   
   Let $\Theta = \begin{pmatrix}
       \theta_1 \\
       \theta_2
   \end{pmatrix}$ be an $h$-approximable vector ($2 \times 1$-matrix).
   
   There exists such an affine line $\mathcal{A} \subset \mathbb{R}^2$ that 
       $$
       \widehat{{ \bf BA}}_{2, 1}^{\Theta} \left[ f_2 \right]  \cap \mathcal{A} = \mathcal{A};
       $$

   in particular, the set 

   $$
   \widehat{\bf D}_{n,m}^{\Theta} \left[ h \right] \cap \mathcal{A}
   $$

   is empty.
\end{remark}

Indeed: let us show that we can find such an $\mathcal{A}$ of the form $\mathcal{A}_{\eta_1} = \{ ( \eta_1, \eta): \,\, \eta \in \mathbb{R} \}$, where we can take $\eta_1$ from a HAW subset of $\mathbb{R}$. 

By Dirichlet's theorem, $\theta_1 = \theta_1^{\top}$ is $f_1$-Dirichlet. Then, by Theorem \ref{PDtoBA_HAW} the set 
$$
\{  \eta_1 \in \mathbb{R}: \,\, (\theta_1, \eta_1) \in \widehat{{ \bf BA}}_{1, 1}[f_1] \}
$$

is HAW. Fix such an $\eta_1$; it is easy to see that, just by a simple observation that 

$$
|| \Theta q - \pmb{\eta} ||_{\mathbb{Z}^2} \geq || \theta_1 q - \eta_1 ||_{\mathbb{Z}} \,\,\,\,\,\,\,\,\,\,\,\, \text{for any } \,\,\,\,\,\,\,\, \pmb{\eta} = (\eta_1, \eta)^{\top},
$$

for any $\eta \in \mathbb{R}$ the vector $(\eta_1, \eta)^{\top}$ belongs to the set $\widehat{{ \bf BA}}_{2, 1}^{\Theta} \left[ f_2 \right]$.

\section{Special cases and corollaries}\label{examplessection}

\subsection{An example: powers and logarithms}\label{powerlog}

We want to illustrate our theorems by applying them to a family of examples. Let ${\bf a} = (a, b) \in \mathbb{R}_+ \times \mathbb{R}$ and consider functions of the form

\begin{equation}\label{standardf}
    f_{\bf a} = f_{(a, b)} = T^{-a} \left( \log T \right)^{-b}.
\end{equation}

Then, according to (\ref{relationg(f)}), the inverse function of $f_{\bf a}$ is equal to $g(T) = (1 + o(1)) \cdot g_{\bf a}(T)$, where

\begin{equation}\label{standardg}
    g_{\bf a}(T) = a^{- \frac{b}{a}} T^{- \frac{1}{a}} \left( \log T \right)^{\frac{b}{a}}.
\end{equation}

Let us also notice that due to Proposition \ref{exp} \ref{prop4A}, for any $\delta > 0$ one can take $\lambda(T) = \left( \log T \right)^{-m(1 + \delta)}$ in Theorem \ref{fullme}, and in this case the function $\tilde{g}$ in Theorem \ref{fullme} satisfies

\begin{equation}\label{standardtildeg}
\tilde{g}(T) \asymp \tilde{g}_{{\bf a}, \delta}(T) =   T^{- \frac{1}{a}} \left( \log T \right)^{\frac{b}{a} - \frac{(m+n)(1 + \delta)}{a}} = \left( \log T \right)^{- \frac{(m+n)(1 + \delta)}{a}} \cdot g_{\bf a}(T).
\end{equation}


\vskip+0.3cm


As an example of our theorems, we can obtain the following results for uniform inhomogeneous approximation:

\begin{itemize}
    \item Suppose $\Theta^{\top}$ is not $f_{\bf a}$-approximable, and let $\delta > 0$. Then, for any $\pmb{\eta} \in \mathbb{R}^n$ the pair $(\Theta, \pmb{\eta})$ is $(1 + \delta) \kappa^{\frac{m}{a}} g_{\bf a}$-Dirichlet. (Follows from Theorem \ref{thmA}).

    \item Suppose $\Theta^{\top}$ is $f_{\bf a}$-approximable. For $\varepsilon$ small enough, the set
    $$
    \widehat{\bf D}_{n,m}^{\Theta} \left[ \varepsilon  g_{\bf a}\right]
    $$

    has measure zero. In addition, the set of $g_{\bf a}$-nonsingular shift vectors is HAW. (Follows from Theorem \ref{thm10} and Theorem \ref{AtoSING_HAW}).
\end{itemize}

In a similar vein, we show some results for asymptotic inhomogeneous approximation:

\begin{itemize}
    \item Suppose $\Theta^{\top}$ is not $f_{\bf a}$-Dirichlet, and let $\delta > 0$. Then, for any $\pmb{\eta} \in \mathbb{R}^n$ the pair $(\Theta, \pmb{\eta})$ is $(1 + \delta) \kappa^{\frac{m}{a}} g_{\bf a}$-approximable. However, the set
    $$
    \widehat{\bf BA}_{n,m}^{\Theta}[g_{\bf a}]
    $$
    of $g_{\bf a}$-badly approximable shift vectors has measure zero. (Follows from Theorem \ref{nonSingtoBA_all} and Theorem \ref{nonSingtoBA_measure}).

    \item Suppose $\Theta^{\top}$ is $f_{\bf a}$-Dirichlet, and let $\delta > 0$. Then the set 
    $$
    \widehat{ \bf W}_{n,m}^{\Theta} \left[ \tilde{g}_{{\bf a}, \delta} \right]
    $$

    of $\tilde{g}_{{\bf a}, \delta}$-approximable shift vectors has measure zero. (Follows from Theorem \ref{fullme}).
\end{itemize}

\vskip+0.3cm

\subsection{Criteria of bad and very good approximability}\label{badandvwa}

In a recent paper \cite{keith}, the following two criteria for bad approximability and very good approximability were presented:

\begin{thm}(Theorem 2.1 from \cite{keith})\label{k_bad}
    A matrix $\Theta$ is badly approximable if and only if there exists $C>0$ such that 
    $$
    \widehat{\bf D}_{n,m}^{\Theta} [Cf_1] = \mathbb{R}^n.
    $$
\end{thm}

\begin{thm}(Theorem 2.3 from \cite{keith})\label{k_vwa}
    If $\Theta$ is not very well approximable, then for any $\varepsilon > 0$

    $$
    \widehat{\bf D}_{n,m}^{\Theta} [f_{1-\varepsilon}] = \mathbb{R}^n.
    $$

    If $\Theta$ is very well approximable, then there exists $\varepsilon > 0$ such that the set

    $$
    \left( \widehat{\bf D}_{n,m}^{\Theta} [f_{1-\varepsilon}] \right)^c
    $$

    is HAW.
\end{thm}

We note that both Theorem \ref{k_bad} and Theorem \ref{k_vwa} directly follow from results in this paper; one just needs to notice that $\Theta$ is badly approximable if and only if $\Theta^{\top}$ is badly approximable, and the same holds for very good approximability. As an illustration, let us show the deduction of Theorem \ref{k_vwa}:

\begin{proof}
    Suppose $\Theta$, and thus $\Theta^{\top}$, is not a very well approximable matrix; that is, for any $\delta > 0$

    $$
    \Theta^{\top} \notin {\bf W}_{m,n}[f_{1+\delta}].
    $$

    Fix $\varepsilon > 0$, and take $\delta$ such that $\frac{1}{1+2\delta} = 1 - \varepsilon$. According to Theorem \ref{thmA}, in this case 

    $$
    \mathbb{R} = \widehat{\bf D}_{n,m}^{\Theta} \left[ f_{\frac{1}{1 + \delta}} \circ \frac{1}{\kappa^m} \right] \subseteq \widehat{\bf D}_{n,m}^{\Theta} \left[ f_{\frac{1}{1 + 2\delta}} \right] = \widehat{\bf D}_{n,m}^{\Theta} \left[ f_{1 - \varepsilon} \right].
    $$

    Now suppose $\Theta$ is very well approximable; therefore, there exists $\delta > 0$ for which $\Theta^{\top} \in {\bf W}_{m,n}[f_{1+\delta}]$. By Theorem \ref{AtoSING_HAW} then the set $\left( \widehat{ \bf Sing}_{n,m}^{\Theta} \left[ f_{\frac{1}{1+\delta}} \right] \right)^c$ is HAW. A simple observation that 
    $$\widehat{ \bf Sing}_{n,m}^{\Theta} \left[ f_{\frac{1}{1+\delta}} \right] \subseteq \widehat{ \bf D}_{n,m}^{\Theta} \left[ f_{\frac{1}{1+ \frac{\delta}{2}}} \right]
    $$

    and identification $1 - \varepsilon = \frac{1}{1+ \frac{\delta}{2}}$ implies that the set 

$$
\left( \widehat{ \bf D}_{n,m}^{\Theta} \left[ f_{1 - \varepsilon} \right] \right)^c \supseteq \left( \widehat{ \bf Sing}_{n,m}^{\Theta} \left[ f_{\frac{1}{1+\delta}} \right] \right)^c
$$

is HAW, which completes the proof.
    
\end{proof}

\vskip+0.3cm 

We note that Theorem \ref{k_bad} and Corollary \ref{x1} provide two criteria of bad approximability with two opposite ideas, which together provide a characterization worth mentioning:

{\it A badly approximable matrix $\Theta$ can be characterized by the following property:

There exist such constants $0 < \varepsilon < C$ that any pair $(\Theta, \pmb{\eta})$ which is not trivially singular belongs to the set

$$
\widehat{ \bf D}_{n,m}^{\Theta} \left[C f_{1} \right] \setminus \widehat{ \bf D}_{n,m}^{\Theta} \left[\varepsilon f_{1} \right].
$$}

\subsection{On a result by Bugeaud and Laurent}\label{blsection}

If we go further and ignore logarithms, we obtain some results for Diophantine exponents. The following proposition directly follows form our results (namely, from Theorem \ref{thmA}, Theorem \ref{thm10}, Theorem \ref{nonSingtoBA_all} and Theorem \ref{fullme}):

\begin{proposition}\label{theoremofbl}
    Let $\Theta \in {\bf M}_{m,n}$. For any $\pmb{\eta} \in \mathbb{R}^n$, 

    \begin{equation} \label{blresult}
    \omega(\Theta, \pmb{\eta}) \geq \frac{1}{\hat{\omega}(\Theta^{\top})} \,\,\,\,\,\,\,\,\,\,\,\, \text{and} \,\,\,\,\,\,\,\,\,\,\,\, \hat{\omega}(\Theta, \pmb{\eta}) \geq \frac{1}{\omega(\Theta^{\top})}.
    \end{equation}

    Moreover, let $\mathcal{A}$ be a non-$\Xi$-exceptional affine subspace of $\mathbb{R}^n$ for some $0 < \Xi < 1$. Then, for Lebesgue-almost all $\pmb{\eta} \in \mathcal{A}$, equality in (\ref{blresult}) holds.
\end{proposition}

Proposition \ref{theoremofbl} (in the special case $\mathcal{A} = \mathbb{R}^n$) was shown by Bugeaud and Laurent in \cite{BL}. We note that the methods we used to prove Theorem \ref{D_measure} generalize the ideas from \cite{BL}.

\subsection{ On a result by Kleinbock and Wadleigh }\label{kwadsection}

In \cite{kwad}, Kleinbock and Wadleigh proved the following statement:

\begin{thm}\label{kwadresult}
    Let $g(z), \,\,\, z \geq 1$ be a non-increasing real valued function. The set $\widehat{\bf D}_{n,m} \left[ g \right]$ has zero measure if the series

    $$
    \sum\limits_{j=1}^{\infty} \frac{1}{j^2 g(j)}
    $$

    diverges, and full measure if it converges.
\end{thm}

The spirit of this theorem is similar to the classical Khintchine–Groshev Theorem in asymptotic homogeneous Diophantine approximations:

\begin{thm}{(Khintchine–Groshev theorem, \cite{G38})}
    Let $f(t), \,\,\, t \geq 1$ be a non-increasing real valued function. The set ${ \bf W}_{n, m}\left[f \right]$ has zero measure if the series

    $$
    \sum\limits_{k=1}^{\infty} f(k)
    $$

    converges, and full measure if it diverges.
\end{thm}

One can show that our Theorem \ref{thm10} together with the Khintchine–Groshev theorem provide an immediate proof of Theorem \ref{kwadresult}, and vice versa: Khintchine–Groshev theorem follows from Theorem \ref{kwadresult} via Theorem \ref{thm10}. The detailed proof of this equivalence in a more general case of approximations with general weight functions can be found in \cite{N25}; here we will just sketch an idea of the argument.

\vskip+0.3cm

First, one can check that it is enough to show Theorem \ref{kwadresult} for strictly decreasing continuous functions $g$: the general case is an easy corollary. From now on, assume $g$ is strictly decreasing and continuous. Let $f$ and $g$ be related via (\ref{relationg(f)}). In Theorem \ref{kwadresult}, we are only interested in the values of $g$ at integer points, so we can assume that $g$ (and thus $f$ as well) is a differentiable function. 

    Let us notice that by the change of variables $t = \frac{1}{g(z)} = f^{-1} \left( \frac{1}{z} \right)$ the integrals

    $$
    \int\limits_1^{+\infty} \frac{dz}{z^2 g(z)} \,\,\,\,\,\,\,\,\,\,\,\,\,\,\,\, \text{and} \,\,\,\,\,\,\,\,\,\,\,\,\,\,\,\, \int\limits_1^{+\infty} f(t) dt
    $$

    either both converge or both diverge. Notice also that the convergence (or divergence) of these integrals is invariant under multiplication of $f$ or $g$ by a constant and under a linear change of variables. Let us fix $\varepsilon < \frac{1}{4}$.

    Using the equivalence of integrals above (and thus the corresponding series), one can see:

    \begin{itemize}
        \item Suppose $\sum\limits_{j=1}^{\infty} \frac{1}{j^2 g(j)} = \infty$; then, by Khintchine–Groshev theorem, for almost every $\Theta \in { \bf M}_{n,m}$, its transposed matrix $\Theta^{\top} \in M_{m,n}$ belongs to ${ \bf W}_{m, n}\left[\delta f (K T) \right]$ for any constants $\delta, K > 0$. Applying Theorem \ref{thm10} with a right choice of $\delta$ and $K$, one can show that 
        $m \left( \widehat{ \bf D}_{n,m}^{\Theta} \left[ g \right]  \right) = 0
        $ for almost every $\Theta \in {\bf M}_{n,m}$.
        The set $\widehat{ \bf D}_{n,m} \left[ g \right]$ is Borel and thus measurable. Fubini's theorem applied to the characteristic function of this set shows that
        $
        m \left( \widehat{ \bf D}_{n,m} \left[ g \right]  \right) = 0.
        $
        \item Suppose $\sum\limits_{j=1}^{\infty} \frac{1}{j^2 g(j)} < \infty$; then, by Khintchine–Groshev theorem, for any $C > 0$ we have 
        $m \left( { \bf W}_{m, n}\left[C f ( T) \right] \right) = 0, 
        $
        which means that
        $
        \Theta^{\top} \in { \bf BA}_{n, m}\left[ f(T) ; C \right]
        $
         for almost every $\Theta \in M_{n,m}$.
        By Theorem \ref{thmA} then, for almost every $\Theta \in M_{n,m}$ we have that
        $
        \widehat{\bf D}_{n,m}^{\Theta} \left[ g \left( \frac{C}{\kappa^m} T \right) \right] = \mathbb{R}^n.
        $
        We take $C = \kappa^m$ to conclude that $\widehat{\bf D}_{n,m}^{\Theta} \left[ g \right] = \mathbb{R}^n$. It remains to apply Fubini's Theorem again to complete the proof.
    \end{itemize}

\vskip+0.3cm

A similar simple argument provides the proof in the other direction: namely, that Theorem \ref{kwadresult} together with Theorem \ref{thmA} and Theorem \ref{thm10} imply Khintchine-Groshev's theorem.

\section{Best approximations and irrationality measure function} \label{bestapproximations}

\subsection{Irrationality measure function}\label{irrmeassection}

 We consider the ordinary irrationality measure function
$$
\psi_\Theta (t)= \min_{{\bf q} = (q_1,...,q_m) \in \mathbb{Z}^m:\,0<|{\bf q}|\le t} ||\Theta{\bf q}||_{\mathbb{Z}^n},
$$
and the irrationality measure function for inhomogeneous approximations
$$
\psi_{\Theta, \pmb{\eta}} (t)= \min_{{\bf q} \in \mathbb{Z }^m:\,0<|{\bf q}|\le t} ||\Theta{\bf q}-\pmb{\eta}||_{\mathbb{Z}^n}.
$$ 

Let us notice that

$$
\Theta^{\top}_j({\bf x}) = \sum_{i=1}^n \theta_{i,j} x_i, \quad 1 \leq j \leq m,
$$

and 

\[
\psi_{\Theta^{\top}}(t) = \min_{(x_1,\dots,x_n)\in\mathbb{Z}^n : 0<|{\bf x}| \leq t}
\max_{1 \leq j \leq m} \left\| \Theta_j^{\top}({\bf x}) \right\| = \min_{{\bf x} = (x_1,...,x_n) \in \mathbb{Z}^n:\,0<|{\bf x}|\le t} ||\Theta^{\top}{\bf x}||_{\mathbb{Z}^m}.
\]

Let us also notice that we can equivalently define Diophantine exponents using the irrationality measure function. It is easy to see that
$$
\omega (\Theta) = \sup\left\{ \gamma\in \mathbb{R}_+:\,
\liminf_{t\to \infty} t^\gamma \cdot \psi_\Theta (t) <\infty\right\}, \,\,\,\,\,\,
\omega (\Theta,\pmb{\eta}) = \sup\left\{ \gamma\in \mathbb{R}_+:\,
\liminf_{t\to \infty} t^\gamma \cdot \psi_{\Theta,\pmb{\eta}} (t) <\infty\right\},
$$

and

$$
\hat{\omega} (\Theta) = \sup\left\{ \gamma\in \mathbb{R}_+:\,
\limsup_{t\to \infty} t^\gamma \cdot \psi_\Theta (t) <\infty\right\},
\,\,\,\,
\hat{\omega} (\Theta,\pmb{\eta}) = \sup\left\{ \gamma\in \mathbb{R}_+:\,
\limsup_{t\to \infty} t^\gamma \cdot \psi_{\Theta,\pmb{\eta}} (t) <\infty\right\}.
$$

\subsection{Best approximations: notation and general properties}\label{bestappsection}
If $\Theta$ is not trivially singular, there exists a sequence
$${\bf p}_\nu \in \mathbb{Z}^m, \,\,\, P_\nu = |{\bf p}_\nu |,\,\,\, {\bf a}_\nu \in \mathbb{Z}^n, \,\,\,\, P_1=1<P_2<...<P_\nu<P_{\nu+1}<... ,
$$
$$ \psi_{\Theta} (t) = \psi_{\Theta} (P_\nu) =||\Theta {\bf p}_\nu||_{\mathbb {Z}^n} =  |\Theta {\bf p}_\nu - {\bf a}_\nu|,\,\,\, \,\,\,\, \text{for}\,\,\,\, P_\nu \le t <P_{\nu+1}
$$ 
of  the best approximations.

\begin{proposition} \label{exp}
Let the values $A$ and $B$ be defined in (\ref{ab})
    Then for  any $\Theta$ which is not trivially singular and for any $\nu$ one has 
\begin{enumerate}[label=(\alph*)]
    \item \label{prop4A} $P_{\nu + A} \geq 2 P_{\nu}$ and
    \item \label{prop4B} $\psi_{\Theta}  (P_{\nu + B}) \leq \frac{1}{2} \psi_{\Theta}  (P_\nu)$.
\end{enumerate}
\end{proposition}

\begin{proof}
  The first statement can be found in \cite{BL}. The second statement is quite similar, but we do not know where it is documented. So we give a proof here. This proof generalizes the argument from \cite{mV}.
Indeed, let us consider  boxes
$$
\Pi_1 = \{({\bf x}, {\bf y}): {\bf x}\in \mathbb{R}^m, |{\bf x}|\le P_\nu,\,\,
{\bf y}\in \mathbb{R}^n, |{\bf y}| \le  \psi_\Theta (P_\nu)\}
$$
and
$$
\Pi_2 = \{({\bf x}, {\bf y}): {\bf x}\in \mathbb{R}^m, |{\bf x}|\le P_\nu,\,\,
{\bf y}\in \mathbb{R}^n, |{\bf y}| \le  \psi_\Theta (P_\nu)/2\}
\subset \mathbb{R}^d.
$$
By the definition of the best approximations there are no non-zero integer points strictly inside $\Pi_1$.
The set $ \Pi_1\setminus\Pi_2$ can be covered by exactly
$B_1 = 2^{2n} - 2^{n}$ shifts of the box
$$
\Pi_0 = \{({\bf x}, {\bf y}): {\bf x}\in \mathbb{R}^m, |{\bf x}|\le P_\nu,\,\,
{\bf y}\in \mathbb{R}^n, |{\bf y}| \le  \psi_\Theta (P_\nu)/4\},
$$
while the box $ \Pi_2$ is covered by exactly $2^n$ shifts of 
$\Pi_0$.
Consider the remainders
\begin{equation}\label{ren}
\pmb{\xi}_{\nu+j} = \Theta{\bf p}_{\nu+j} -{\bf{a}}_{\nu+j},\,\,\, 0\le j \le B.
\end{equation}
If at least one of the remainders (\ref{ren}) belongs to 
$\Pi_2$, everything is proven. So we may suppose that all the 
remainders (\ref{ren}) belong to $ \Pi_1\setminus\Pi_2$.
By the pigeonhole principle, there exists a shift $\Pi_0^*$ of $\Pi_0$ which contains at least  $\frac{B}{B_1}+1 = 2^{m-1}+1$ different remainder vectors 
$$
\pmb{\xi}_{\nu+j_k}, \,\,  1\le k \le 2^{m-1}+1
$$
of the form (\ref{ren}).
For all these vectors we have 
$$
|{\bf p}_{j_k}|\le  P_{\nu+B} = 
\max_{0\le j \le B}P_{\nu+j}.
$$
Now we cover   the box
$$
\mathcal{P} = \{ {\bf x}\in \mathbb{R}^m, |{\bf x}|\le P_{\nu+B}\}\subset \mathbb{R}^m
$$
by 
$2^m$ shifts of the box
$$
\mathcal{P}_0 = \{ {\bf x}\in \mathbb{R}^m, |{\bf x}|\le P_{\nu+B}/2\}\subset \mathbb{R}^m.
$$
Applying the pigeonhole principle one more time to the collection of vectors
\begin{equation}\label{ren1}\pm p_{j_k},\,\,\,  1\le k \le 2^{m-1}+1,
\end{equation}
we see that there are at least two integer vectors of the form 
(\ref{ren1}), say 
$ {\bf p}_{\nu+j'} $ 
 {and}
$ 
{\bf p}_{\nu+j''} 
$, which belong to the same shift of $\mathcal{P}_0$, and thus 
$$
|{\bf p}_{\nu+j'}- {\bf p}_{\nu+j''}|\le P_{\nu+B} = 
\max_{0\le j \le B}P_{\nu+j}.
$$
If we put ${\bf p }={\bf p}_{\nu+j'}- {\bf p}_{\nu+j''}$, we see that 
simultaneously
$$
|{\bf p }|\le P_{\nu+B}\,\,\,
\text{and}\,\,\,
||\Theta{\bf p}|| \le \psi_\Theta (P_\nu)/2.
$$
By the definition of the best approximations this means that 
$\psi_\Theta (P_{\nu+B})/2\le \psi_\Theta (P_\nu)/2$.

\end{proof}

If the pair $(\Theta,\pmb{\eta})$ is not trivially singular, there exists a sequence
$${\bf q}_\nu \in \mathbb{Z}^m, \,\,\, Q_\nu = |{\bf q}_\nu |, \,\,\,\, Q_1=1<Q_2<...<Q_\nu<Q_{\nu+1}<... ,
$$
of  the best inhomogeneous approximations:

$$ \psi_{\Theta,\pmb{\eta}} (t) = \psi_{\Theta,\pmb{\eta}} (Q_\nu) =||\Theta {\bf q}_\nu -\pmb{\eta}||_{\mathbb {Z}^n},\,\,\, \,\,\,\, \text{for}\,\,\,\, Q_\nu \le t <Q_{\nu+1}.
$$

Throughout the paper, we will keep this notation ${\bf q}_{\nu}$ for the best approximation vectors of the pair $(\Theta,\pmb{\eta})$. We will also consistently use the notation ${\bf p}_{\nu}$ for the best approximation vectors of a matrix $\Theta$, given it is not trivially singular, and will let $P_\nu = |{\bf p}_{\nu}|$. We will use the notation ${\bf y}_{\nu}$ and $Y_\nu = |{\bf y}_{\nu}|$ analogously for the best approximation vectors of the transposed matrix $\Theta^{\top}$ and their magnitudes.

\subsection{Asymptotic directions}\label{asympd}

In Theorem \ref{AtoSING_HAW} and Theorem \ref{PDtoBA_HAW}, the condition on an affine subspace $\mathcal{A}$ of $\mathbb{R}^n$ depends on the set of all the asymptotic directions of the best approximation vectors of $\Theta^{\top}$. In this subsection, we give all the definitions and brief overview of the subject.

Let $\Theta^{\top}$ be a non trivially singular matrix, and ${\bf y}_{\nu}$ its sequence of the best approximation vectors. Let $S^{n-1}$ be the $(n-1)$-dimensional sphere in $\mathbb{R}^n$ centered at zero. We say $ \zeta$ is an {\it asymptotic direction} for the best approximation sequence ${\bf y}_{\nu}$ if $\zeta$ is a limit point of the set $\left\{ \pm\frac{{\bf y}_{\nu}}{|{\bf y}_{\nu}|}: \,\, \nu \in \mathbb{N} \right\}$. We will denote the set of all the asymptotic directions by $\Omega$, and will refer to it briefly just as "the set of asymptotic directions".

If $\Theta^{\top}$ is trivially singular, we will define the set of asymptotic directions as an empty set.

The sets of asymptotic directions were studied before; see \cite{Ge04} for the results concerning the asymptotic direction sets in the case of linear forms. It is easy to see that any set of asymptotic directions is closed and centrally symmetric. It is shown in \cite{Ge04} that in the special case $m=1$ there are no other restrictions:

\begin{proposition}{(Theorem 4 in \cite{Ge04})}
    Let $\Omega \subseteq S^{n-1}$ be a closed centrally symmetric set. Then, the set of $1 \times n$ matrices $\Theta^{\top}$ for which $\Omega$ is the set of asymptotic directions has the cardinality of continuum.
\end{proposition}

In some special cases, we can construct linear forms with a prescribed set of asymptotic directions and in addition control their approximation properties. Namely, Theorem 2 from \cite{Ge20} guarantees that for any large enough $\gamma$ there exists a $\Theta^{\top} \in M_{1,n}$ such that $\omega(\Theta^{\top}) = \gamma$ and the set $\Omega$ of asymptotic directions of $\Theta^{\top}$ consists of two pairs of centrally symmetric points. It is easy to see from the proof that we can control the uniform exponent of $\Theta^{\top}$ or, more generally, a particular function $f$ for which $\Theta^{\top}$ is $f$-approximable/$f$-Dirichlet/etc. instead. This provides us with examples for which Theorem \ref{AtoSING_HAW} and Theorem \ref{PDtoBA_HAW} hold for almost every $\mathcal{A}$.

\section{Some lemmata about measure and dimension}\label{measureanddim}

\subsection{Exceptional sets} \label{exceptional}

In this section, we introduce some useful language and prove simple but useful statements. We use the notation $\measuredangle({\bf u, v})$ for the angle between two vectors ${\bf u, v} \in \mathbb{R}^n$.

First, assume $\Theta^{\top}$ is not trivially singular, and let $\Xi > 0$. Let us call a point ${ \bf v} \in \mathbb{R}^n$ on the unit sphere {\it $\Xi$-exceptional} (for the matrix $\Theta^{\top}$) if 

\begin{equation}\label{xxExe}
\text{for any}\,\, \delta > 0
\,\,
\text{
there exist infinitely many } \,\,\ \nu \in \mathbb{N}
\,\,
\text{for which}\,\, \measuredangle({ {\bf y}_{\nu}, { \bf v}})  < \frac{\delta}{|{\bf y}_{\nu}|^{\Xi}}.
\end{equation}

We will call an affine subspace $\mathcal{A}$ of $\mathbb{R}^n$ with $\mathcal{L}$ as its corresponding linear subspace {\it $\Xi$-exceptional} if $\mathcal{L}^{\perp}$ intersects the set of $\Xi$-exceptional points.

\vskip+0.3cm

We will treat the case of trivially singular $\Theta^{\top}$ separately. Suppose ${\bf 0 \neq  z} \in \mathbb{Z}^n$ is such that $\Theta^{\top} {\bf z} = {\bf 0}$; we will say that the exceptional subspaces are the subspaces of the form

\begin{equation}\label{trivexceptional}
\{ \pmb{\eta} \in \mathbb{Z}^n: \,\,\,  \pmb{\eta} \cdot {\bf z} = a  \} \,\,\,\,\,\,\,\,\,\,\,\, \text{for ${\bf a \in \mathbb{Z}}$}
\end{equation}

or the ones that are contained in subspaces of the form (\ref{trivexceptional}). For convenience, we will use a convention that exceptional sets are $\Xi$-exceptional for any $\Xi$.

\vskip+0.3cm

We prove our main theorems for subspaces that are not exceptional. Here, we aim to show that the sets of $\Xi$-exceptional subspaces are very small.

\begin{lemma}\label{zerodim}
    The set of $\Xi$-exceptional points on the unit sphere has zero Hausdorff dimension for any $\Xi > 0$.
\end{lemma}

\begin{proof}
    Let us denote the set of $\Xi$-exceptional points by $E$. Then,

    $$
    E = \bigcap\limits_{\delta > 0} \bigcup\limits_{\nu =1}^{\infty} B_{\frac{\delta}{|{\bf y}_\nu|^{\Xi}}} \left( \frac{{\bf y}_\nu}{|{\bf y}_\nu|} \right).
    $$

   Fix $s > 0$. Let $\delta > 0$ and let $C(s) : = \sum\limits_{\nu = 1}^{\infty} \left( \frac{\delta}{|{\bf y}_\nu|^{\Xi}} \right)^s$. Then, 

   $$
   \mathcal{H}^s_{\delta}(E) \leq \delta^s C(s)  \xrightarrow[\delta \to 0]{} 0,
   $$

   where the limit being zero is provided by the first statement of Proposition \ref{exp}. This shows that the $s$-dimensional Hausdorff measure of $E$ is zero and completes the proof.
\end{proof}

\begin{lemma}\label{dimensionlemmaaff}
    Let $E$ be a nonempty subset of $\mathbb{R}^l$ of Hausdorff dimension $0$, and let $X_E$ be the set of $r$-dimensional affine subspaces of $\mathbb{R}^l$ intersecting with $E$. Then, $\dim_H X_E = r(l-r)$. 
\end{lemma}

\begin{proof}
    Let $e \in E$; then the set of all the $r$-dimensional subspaces containing $e$ can be identified with $Gr(l, r)$. Let us consider the set $P_E = \{ (e, V): \,\,\, e \in E, \, e \in V, \, V \in X_{\{ e \}} \}$: it can be identified with $E \times Gr(l, r)$. Recall that the Hausdorff and upper box dimensions for the set $Gr(l, r)$ coincide (since the Grassmannian is a manifold); therefore, by a classical product formula (see for example \cite{falconer}, Corollary 7.4) $\dim_H P_E = \dim_H Gr(l, r) + \dim_H E = r(l-r)$. The set $X_E$ is the projection of $P_E$ on the second coordinate, so $\dim_H X_E \leq \dim_H P_E = r(l-r)$.

    On the other hand, let us notice that the set $X_{\{ e \}}$ of all $r$-dimensional subspaces containing a single point $e$ can be identified with $Gr(l, r)$ and thus has dimension $r(l-r)$. But $E$ is nonempty, hence $X_E$ contains a set of the form $X_{\{ e \}}$ and therefore $\dim_H X_E \geq \dim_H X_{\{ e \}} = r(l-r)$, which completes the proof.
\end{proof}

The next statement immediately follows from the local identification of the unit sphere in $\mathbb{R}^n$ with $\mathbb{R}^{n-1}$, Lemma \ref{zerodim} and Lemma \ref{dimensionlemmaaff}.

\begin{corollary}\label{dimensionlemma}
    Let $X_E$ be the set of $k$-dimensional linear subspaces of $\mathbb{R}^n$ intersecting with the set $E$ of $\Xi$-exceptional points for $\Theta^{\top}$. If $E$ is not empty, then $\dim_H X_E = (k-1)(n-k)$.
\end{corollary}

Corollary \ref{dimensionlemma} shows that the linear subspaces $\mathcal{L}$ corresponding to $\Xi$-exceptional $k$-dimensional affine subspaces $\mathcal{A}$ form a $\leq (k-1)(n-k)$-dimensional subset on the grassmanian $Gr(n, k)$ for any $\Xi > 0$. In particular, the set of $\Xi$-exceptional $k$-dimensional affine subspaces has measure zero.

\subsection{Statements about measure}\label{measuretechnical}

The following lemmata will be used in proofs of our main statements. We denote the orthogonal projection of a vector ${\bf u}$ on a linear subspace $\mathcal{L}$ of $\mathbb{R}^n$ by $pr_{\mathcal{L}}{\bf u}$.

\begin{lemma}\label{lemmameasures}

Let $\mathcal{A}$ be an affine subspace of $\mathbb{R}^n$ with $\mathcal{L}$ as the corresponding linear subspace, and let ${ \bf u }$ be a vector in $\mathbb{R}^n$. We will denote by $U(\sigma)$ the set $ \{ \pmb{\eta} \in \mathcal{A}: \,\,\,||\pmb{\eta} \cdot {\bf u}|| \leq \sigma \}$.

There exists a function $\xi(t)$ satisfying $\lim\limits_{t \rightarrow \infty} \xi(t) = 0$ such that for any $\frac{1}{2} > \sigma > 0$ one has

$$
\frac{m \left( U(\sigma) \cap \left( B_M(0) + { \bf {\bf a}} \right) \right)}{
 m (
B_M(0))} = 2 \sigma \cdot (1 + \xi({M} \cdot{|pr_{\mathcal{L}}{\bf u}|})),
$$
where $B_M(0)$ denotes the ball (with respect to the sup-norm in $\mathbb{R}^d$) of radius $M$ 
intersected with $\mathcal{L}$, $m$ is $(\dim \mathcal{L})$-dimensional Lebesgue measure, and ${ \bf {\bf a}} \in \mathcal{A}$ is any vector.
\end{lemma}

\begin{proof}

   Let ${\bf v} \in \mathcal{L}^{\perp}$ be defined by ${\bf u} = pr_{\mathcal{L}}{\bf u} + {\bf v}$. Let $\xi = \{ { \bf {\bf a}} \cdot {\bf u} \} \in [0, 1)$ be the fractional part of scalar product of ${\bf a}$ and ${\bf u}$. Then, for any $\pmb{\eta} = \tilde{\pmb{\eta}} + { \bf {\bf a}} \in \mathcal{A}$,

    $$
    ||\pmb{\eta} \cdot {\bf u}|| = ||\tilde{\pmb{\eta}} \cdot pr_{\mathcal{L}}{\bf u} + { \bf {\bf a}} \cdot {\bf v}|| = ||\tilde{\pmb{\eta}} \cdot pr_{\mathcal{L}}{\bf u} + \xi|| = ||\left( \tilde{\pmb{\eta}} + \xi \frac{pr_{\mathcal{L}}{\bf u}}{|pr_{\mathcal{L}}{\bf u}|^2}\right) \cdot pr_{\mathcal{L}}{\bf u}||.
    $$

    Thus,

    $$
    \frac{m \left( \{ \pmb{\eta} \in \mathcal{A}: \,\,\,||\pmb{\eta} \cdot {\bf u}|| < \sigma \} \cap \left( B_M(0) + { \bf {\bf a}} \right)\right)}{m \left( B_M(0) \right)} = 
    $$
    $$
    \frac{m \left( \{ \tilde{\pmb{\eta}} \in \mathcal{L}: \,\,\,||\left( \tilde{\pmb{\eta}} + \xi \frac{pr_{\mathcal{L}}{\bf u}}{|pr_{\mathcal{L}}{\bf u}|^2}\right) \cdot pr_{\mathcal{L}}{\bf u}|| < \sigma \} \cap  B_M(0)\right)}{m \left( B_M(0) \right)}=
    $$
    $$
    \frac{m \left( \{ \tilde{\pmb{\eta}} \in \mathcal{L}: \,\,\,||\tilde{\pmb{\eta}} \cdot pr_{\mathcal{L}}{\bf u}|| < \sigma \} \cap \left( B_M(0) - \xi \frac{pr_{\mathcal{L}}{\bf u}}{|pr_{\mathcal{L}}{\bf u}|^2} \right)\right)}{m \left( B_M(0) \right)} = 
    $$
    $$
    \frac{m \left( \{ \tilde{\pmb{\eta}} \in \mathcal{L}: \,\,\,||\tilde{\pmb{\eta}} \cdot \frac{pr_{\mathcal{L}}{\bf u}}{|pr_{\mathcal{L}}{\bf u}|}|| < \sigma \} \cap \left( B_{M \cdot |pr_{\mathcal{L}}{\bf u}|}(0) - \xi \frac{pr_{\mathcal{L}}{\bf u}}{|pr_{\mathcal{L}}{\bf u}|} \right)\right)}{m \left( B_{M \cdot |pr_{\mathcal{L}}{\bf u}|}(0) \right)} =  2 \sigma \cdot (1 + o(1)),
    $$

    where the last line is obtained by the scaling $\tilde{\pmb{\eta}} \mapsto |pr_{\mathcal{L}}{\bf u}| \cdot \tilde{\pmb{\eta}}$ and $o(1)$ is infinitely small as the radius of the ball we are considering goes to infinity.
\end{proof}

Fix $M > 0$. Lemma \ref{lemmameasures} shows that the bigger $|pr_{\mathcal{L}}{\bf u}|$ gets, the closer $U(\sigma)$ and any ball of radius $M$ are to independent sets. This provides us with the following useful claim:

\begin{lemma}\label{independencelemma}
Let $\mathcal{A}$ be an affine subspace of $\mathbb{R}^n$ with $\mathcal{L}$ as the corresponding linear subspace, and let the sequence $\{{ \bf u }_{k}\}$ of vectors in $\mathbb{R}^n$ be such that $\lim\limits_{k \rightarrow \infty} \frac{|pr_{\mathcal{L}}{\bf u}_{k }|}{|pr_{\mathcal{L}}{\bf u}_{k+1}|} = 0$. Let 
$$
U_k = U_k(\sigma): = \{ \pmb{\eta} \in \mathcal{A}: \,\,\,||\pmb{\eta} \cdot {\bf u}_k|| \leq \sigma \}.
$$
Fix $\sigma < \frac{1}{2}$. Then, for any $M > 0$ and any ${\bf a} \in \mathcal{A}$, one has 

$$
\lim\limits_{k, s \rightarrow \infty, \,\, k \neq s} \frac{m\left( U_k \cap U_s \cap \left( B_M(0) + { \bf {\bf a}} \right) \right)}{m\left( U_s \cap \left( B_M(0) + { \bf {\bf a}} \right)\right) m\left( U_k\cap \left( B_M(0) + { \bf {\bf a}} \right)\right)} =m \left( B_M(0) \right).
$$
\end{lemma}

\begin{lemma}\label{lemma2bl}
    Let $\omega(t)$ be a continuous decreasing function such that $\sum\limits_{\nu =1}^{\infty} \omega(Y_{\nu}) < \infty$. Let $\mathcal{A}$ be an non-1-exceptional affine subspace of $\mathbb{R}^n $ of positive dimension. For almost all $\pmb{\eta} \in \mathcal{A}$, one has 

    $$
    ||\pmb{\eta} \cdot {\bf y}_{\nu}|| \geq \omega(Y_\nu)
    $$

    for all $\nu$ large enough.
\end{lemma}

\begin{proof}
The property of being not $1$-exceptional guarantees that there exists such a $\delta> 0$ that $|\measuredangle({ {\bf y}_\nu, \mathcal{L}}) - \frac{\pi}{2}| > \frac{\delta}{|{\bf y}_\nu|}$ for all $\nu$ (except maybe finitely many). This guarantees that $|pr_{\mathcal{L}}{\bf y}_{\nu}| > \delta$.
    By Lemma \ref{lemmameasures},

    $$
    \frac{m \left( \{ \pmb{\eta} \in \mathcal{A}: \,\,\,||\pmb{\eta} \cdot {\bf y}_{\nu}|| < \omega(Y_{\nu}) \} \cap \left( B_M(0) + {\bf u} \right)\right)}{m \left( B_M(0) \right)} \leq 2 \omega(Y_{\nu})(1 + \xi(\delta M)).
    $$

 Thus, by Borel-Cantelli lemma, 

    $$
    m \left( \{ \pmb{\eta} \in \mathcal{A}: \,\,\,||\pmb{\eta} \cdot {\bf y}_{\nu}|| < \omega(Y_{\nu})  \,\, \text{for infinitely many} \, \nu\} \right) = 0.
    $$
\end{proof}

\section{Transference principle. Proofs of the main results.}\label{alltransference}

We start with recalling a classical and very useful transference lemma. We will use it in the form of Lemma 3 from \cite{BL}; note that it is equivalent to the statement of Theorem XVII from Chapter V in Cassels' book \cite{c}:

\begin{lem}\label{lemma3bl}
    Let $Y$ and $Q$ be two positive real numbers, such that the inequality

    $$
    ||\Theta^{\top} {\bf y}|| \geq \frac{\kappa}{Q} 
    $$

    holds for any nonzero ${\bf y} \in \mathbb{Z}^n$ with $|{\bf y}| \leq Y$. Then for any $\pmb{\eta} \in \mathbb{R}^n$ there exists such a ${\bf q} \in \mathbb{Z}^m$ with $|{\bf q}| \leq Q$ that 

    $$
    ||\Theta {\bf q} - \pmb{\eta}|| \leq \frac{\kappa}{Y}.
    $$
\end{lem}

Another classical result which will be used in the proofs is Chung-Erd{\"o}s inequality (see \cite{ChEr}; we are using it in form of Lemma 5, §3, Ch.1 in \cite{Spr}).

\begin{lem}\label{Sprin}
    Let $(X, \mu)$ be a measure space, and let $\{ E_k \}_{k=1}^{\infty}$ be a countable collection of measurable sets. Let $E = \bigcap\limits_{N=1}^{\infty} \bigcup\limits_{k=N}^{\infty} E_k$. Suppose $\sum\limits_{k=1}^{\infty} \mu(E_k) = \infty$. Then, 

    $$
    \mu(E) \geq \limsup\limits_{N \rightarrow \infty} \frac{\left( \sum\limits_{k=1}^{N} \mu(E_k) \right)^2}{\sum\limits_{k,s=1}^{N} \mu(E_k \cap E_s)}.
    $$
\end{lem}

\subsection{Some general theorems via transference}\label{transferencesection}

In this section we will formulate and prove general results which imply Theorems \ref{thmA}, \ref{AtoSING_HAW},  \ref{nonSingtoBA_all}, \ref{PDtoBA_HAW} and \ref{fullme}. The deduction is provided by setting

\begin{equation} \label{notationchange}
f(T) = \psi^m \left( T^{\frac{1}{n}} \right) \,\,\,\,\,\,\,\, \text{and} \,\,\,\,\,\,\,\,\,\,\,\, g(T) = \frac{1}{\rho^{n}\left( T^{\frac{1}{m}}\right)}.
\end{equation}


To make the paper self-contained, we will formulate and prove a statement which implies Theorem \ref{thmA} and Theorem \ref{nonSingtoBA_all}.

\begin{theorem}\label{Jarnik}
 Let $\psi$ be a continuous function decreasing to zero as $t \rightarrow + \infty$, and such that

\begin{itemize}
    \item[(a)] The inequality
    \begin{equation} \label{limsup1}
    {\psi_{\Theta^{\top}}(t)} \geq {\psi(t)}
    \end{equation}

    holds for an unbounded set of $t \in \mathbb{R}_+$,  or
\item[(b)] the inequality
    \begin{equation} \label{liminf1}
     {\psi_{\Theta^{\top}}(t)} \geq {\psi(t)}
    \end{equation} 

    holds for all $t \in \mathbb{R}_+$ large enough.
\end{itemize}
    Let 
    \begin{equation}\label{xxrho}
    \rho(t)\,\,\,\text{be the function inverse to the function}
    \,\,\, t \mapsto 1/\psi(t). 
\end{equation}

Then, for any $\pmb{\eta} \in \mathbb{R}^n$,

\begin{itemize}
\item[(a)] 

$$
\psi_{\Theta, \pmb{\eta}}(t) \cdot \rho \left( \frac{1}{\kappa}t \right) \leq \kappa
$$
 for an unbounded set of $t \in \mathbb{R}_+$, or
   
\item[(b)] 

$$
\psi_{\Theta, \pmb{\eta}}(t) \cdot \rho \left( \frac{1}{\kappa}t \right) \leq \kappa
$$
 for all $t \in \mathbb{R}_+$ large enough.

\end{itemize}
    
\end{theorem}

Theorem \ref{Jarnik} is a consequence of Lemma \ref{lemma3bl}.

{\it Proof of Theorem \ref{Jarnik}.}
     Let $Y>0$ be such that 

    \begin{equation}\label{lemma3blineq}
        \psi_{\Theta^{\top}}(Y) \geq \psi(Y).
    \end{equation}

    Let us define $t$ by $\frac{\kappa}{t} = \psi(Y)$; then $Y = \rho(\frac{1}{\kappa} t)$ and, by Lemma \ref{lemma3bl}, there exists ${\bf q} \in \mathbb{Z}^m$ with $|{\bf q}| \leq t$ such that 

    \begin{equation}\label{lemma3blineq2}
    \psi_{\Theta, \pmb{\eta}}(t) \cdot \rho \left(\frac{1}{ \kappa} t \right) \leq ||\Theta {\bf q} - \pmb{\eta}||\rho\left(\frac{1}{ \kappa} t \right)  \leq \kappa.
    \end{equation}

    \begin{enumerate}[label=(\alph*)]
        \item[(a)] The set of $Y$ such that (\ref{lemma3blineq}) holds is unbounded, therefore so is the set of $t$ for which (\ref{lemma3blineq2}) holds, which completes the proof.

        \item[(b)] Condition (\ref{lemma3blineq}) holds for any $Y$ large enough, so (\ref{lemma3blineq2}) holds for all $t$ large enough, and thus we get the desired statement.
    \end{enumerate} \qed

\vskip+0.3cm

The remainder of this section contains the proof of two results, implying Theorems \ref{AtoSING_HAW}, \ref{PDtoBA_HAW} and \ref{fullme}.

\begin{theorem}\label{Dyakova}

 Let $\psi$ be a continuous function decreasing to zero as $t \rightarrow + \infty$, and such that

\begin{itemize}
    \item[(a)] 
    \begin{equation} \label{limsup11}
    \text{the inequality} \,\, {\psi_{\Theta^{\top}}(t)} \leq {\psi(t)} \,\, \text{holds for all $t \in \mathbb{R}_+$ large enough,}
    \end{equation}

    or
\item[(b)] 
    \begin{equation} \label{liminf11}
     \text{the inequality} \,\, {\psi_{\Theta^{\top}}(t)} \leq {\psi(t)} \,\, \text{holds for an unbounded set of $t \in \mathbb{R}_+$.}
    \end{equation}

\end{itemize}

        Let $\Omega$ be the set of all the asymptotic directions for the best approximations of $\Theta^{\top}$. 
    
        Let $\mathcal{A}$ be an affine subspace of $\mathbb{R}^n$ of positive dimension, with $\mathcal{L}$ as the corresponding linear subspace and $\mathcal{L}^{\perp}$ its orthogonal complement. If $\Theta^{\top}$ is trivially singular, assume also that $\mathcal{A}$ is not exceptional.

        Let $S(\Theta)$ be the set of such $\pmb{\eta} \in \mathbb{R}^n$ for which there exists a constant $K(\pmb{\eta})$, such that 
    \begin{enumerate}
\item[(a)]       $$
        \liminf\limits_{t \rightarrow \infty} \psi_{\Theta, \pmb{\eta}}(t) \cdot \rho(K(\pmb{\eta})t) > 0
        $$

        or

\item[(b)]       $$
        \limsup\limits_{t \rightarrow \infty} \psi_{\Theta, \pmb{\eta}}(t) \cdot \rho(K(\pmb{\eta})t) > 0.
        $$
    \end{enumerate}
        If $\mathcal{L}^{\perp} \cap \Omega = \emptyset$, then the set $S(\Theta) \cap \mathcal{A}$ is HAW in $\mathcal{A}$.

\end{theorem}

\begin{theorem}\label{D_measure}
Let $\Theta^{\top}$ be not trivially singular, and let $\psi$ be a continuous function decreasing to zero as $t \rightarrow + \infty$, such that \eqref{limsup11} holds for all $t \in \mathbb{R}_+$ large enough.

        \noindent Let $\mathcal{A}$ be a non 1-exceptional affine subspace in $\mathbb{R}^n$.

        \noindent Let $\omega(t)$ be a function such that $\sum\limits_{\nu=1}^{\infty} \omega(Y_{\nu}) < \infty$. Let $G(t)$ be the inverse function to $\frac{\omega(t)}{\psi(t)}$, and let $\tilde{\rho}(t): = \frac{G(t)}{\omega \left( G(t) \right)}$. Then, for almost any $\pmb{\eta} \in \mathcal{A}$ the inequality
        $$
        \psi_{\Theta, \pmb{\eta}}(t) \cdot \tilde{\rho}(2mt) \geq \frac{1}{2n}
        $$

        holds for any $t$ large enough.

\end{theorem}

\vskip 0.3 cm

\begin{remark} We note that Theorem \ref{fullme} follows from Theorem \ref{D_measure} via the identifications 
$$
\lambda(T) = \omega^m\left(T^{\frac{1}{n}}\right), \,\,\,\,\,\,\,\, H(T) = G^n \left( T^{\frac{1}{m}}\right), \,\,\,\,\,\,\,\,\, \tilde{g}(T) = \frac{1}{\tilde{\rho}^n \left( T^{\frac{1}{m}} \right)}  
$$

and using the fact that the series $\sum\limits_{\nu = 1}^{\infty} \lambda^{\frac{1}{m}} \left( Y_{\nu}^n \right) $ and $\sum\limits_{\nu = 1}^{\infty} \frac{1}{2m}\lambda^{\frac{1}{m}} \left( Y_{\nu}^n \right)$ converge simultaneously.

\end{remark}

\vskip+0.3cm

\noindent The proofs of Theorem \ref{Dyakova} and Theorem \ref{D_measure} use a transference argument. Let 
$$
F_{\pmb{\eta}}(t): \,\,\,\, \mathbb{R}_{\geq 1} \rightarrow \mathbb{R}_{\geq 0}
$$

\noindent be some family of non-increasing functions, parametrized by $\pmb{\eta} \in \mathbb{R}^n$. We also want $F_{\pmb{\eta}}(t)$ to decrease slower than $\psi(t)$, so that the function $\frac{F_{\pmb{\eta}}(t)}{\psi(t)}$ is increasing to infinity. Let $G_{\pmb{\eta}}(t)$ be the inverse function to $\frac{F_{\pmb{\eta}}(t)}{\psi(t)}$. Let $\tilde{\rho}_{\pmb{\eta}}(t): = \frac{G_{\pmb{\eta}}(t)}{F_{\pmb{\eta}} \left( G_{\pmb{\eta}}(t) \right)}$.

We start with the following useful observation. For any $\pmb{\eta} \in \mathbb{R}^n$ and any ${\bf q} \in \mathbb{Z}^m$ one has

    $$
    \pmb{\eta} \cdot { \bf y}_{\nu} = \sum\limits_{j=1}^m q_j \Theta_j^T({ \bf y}_{\nu}) - \sum\limits_{i=1}^n \left(  \Theta_i({\bf q}) - \eta_i \right) y_{\nu, i},
    $$

    and so 

    $$|| \pmb{\eta} \cdot { \bf y}_{\nu}|| = ||\sum\limits_{j=1}^m q_j \Theta_j^T({ \bf y}_{\nu}) - \sum\limits_{i=1}^n \left(  \Theta_i({\bf q}) - \eta_i \right) y_{\nu, i}|| \leq
    $$

        \begin{equation}\label{scalarprodbound}
        \leq ||\sum\limits_{j=1}^m q_j \Theta_j^T({ \bf y}_{\nu})|| + ||\sum\limits_{i=1}^n \left(  \Theta_i({\bf q}) - \eta_i \right) y_{\nu, i}|| \leq m |{\bf q}| \psi_{\Theta^{\top}} (Y_\nu) + n Y_{\nu} \cdot ||\Theta{\bf q} - \pmb{\eta}||.
        \end{equation}

    \begin{remark}\label{trivsingremark}
        Let us also note that if $\Theta^{\top}$ is trivially singular and $\Theta^{\top} {\bf z} = {\bf 0}$ for some ${\bf0 \neq z} \in \mathbb{Z}^n$,  (\ref{scalarprodbound}) still holds with ${\bf y}_\nu$ replaced by ${\bf z}$. Thus,
        \begin{equation}\label{trivsing_inequ}
            ||\Theta{\bf q} - \pmb{\eta}|| \geq \frac{|| \pmb{\eta} \cdot z||}{n |{\bf z}|}
        \end{equation}

        for any ${\bf q} \in \mathbb{Z}^m$. In particular, if $\pmb{\eta}$ does not belong to an exceptional subspace of $\mathbb{R}^n$, then $||\Theta{\bf q} - \pmb{\eta}|| \geq C > 0$ is bounded by some constant.
    \end{remark}

    \noindent In the remainder of this section we will assume that $\Theta^{\top}$ is not trivially singular. Remark \ref{trivsingremark} shows that the statements of Theorem \ref{Dyakova} and Theorem \ref{D_measure}, as well as the statements of all the theorems in sections \ref{section31} and \ref{section32}, hold in even stronger form for non-exceptional subspaces in the case of trivially singular $\Theta^{\top}$.

\begin{lemma}\label{transference}
    Suppose $\Theta$ and $\psi$ are chosen in such a way that 
    \begin{itemize}
        \item[(a)] \eqref{limsup11} holds or
        \item[(b)] \eqref{liminf11} holds.
    \end{itemize}
    Suppose $\pmb{\eta} \in \mathbb{R}^n$ satisfies the condition $||\pmb{\eta} \cdot { \bf y}_{\nu}|| \geq F_{\pmb{\eta}}(Y_\nu)$ for any $\nu$ large enough. Then the condition
    
    $$||\Theta {\bf q} - \pmb{\eta}|| \cdot \tilde{\rho}(2mQ) \geq \frac{1}{2n} \,\,\,\, \text{for all} \,\,\,\,  {\bf q}  \,\,\,\, \text{with} \,\,\,\, |{\bf q}| \leq Q$$ 
\begin{itemize}
    \item[(a)] holds for all $Q \in \mathbb{R}_+$ large enough or

    \item[(b)]holds for some sequence $\{  Q_{\nu} \}$ of real numbers increasing to infinity.
\end{itemize}
     
\end{lemma}

\begin{proof}

        \begin{itemize}
            \item[(a)]

        Let us fix ${\bf q} \in \mathbb{Z}^m$ with large enough norm, such that $\frac{F_{\pmb{\eta}}(Y_{1})}{\psi(Y_{1})} \leq 2m|{\bf q}|$. Choose $\nu$ such that 
 
      \begin{equation} \label{boundsforq}
        \frac{F_{\pmb{\eta}}(Y_{\nu})}{\psi(Y_{\nu})} \leq 2m|{\bf q}| < \frac{F_{\pmb{\eta}}(Y_{\nu+1})}{\psi(Y_{\nu+1})};
        \end{equation}

        the function $F_{\pmb{\eta}}$ is non-increasing, so 

        \begin{equation}\label{upperbdq}
            2m|{ \bf q}| < \frac{F_{\pmb{\eta}}(Y_{\nu+1})}{\psi(Y_{\nu+1})} \leq \frac{F_{\pmb{\eta}}(Y_{\nu})}{\psi(Y_{\nu+1})}.
        \end{equation}
 
        Let us notice that by (\ref{limsup11}) for any $\nu$ large enough one has $\psi_{\Theta^{\top}} (Y_\nu) \leq \psi(Y_{\nu+1})$.

        Using this inequality, (\ref{scalarprodbound}) and (\ref{upperbdq}), we get that

        $$
        F_{\pmb{\eta}}(Y_\nu) \leq m |{\bf q}| \psi_{\Theta^{\top}} (Y_\nu) + n Y_{\nu} \cdot ||\Theta {\bf q} - \pmb{\eta}|| \leq 
        $$
        $$
        m |{\bf q}| \psi(Y_{\nu+1}) + n Y_{\nu} \cdot ||\Theta {\bf q} - \pmb{\eta}|| \leq \frac{1}{2} F_{\pmb{\eta}}(Y_\nu) + n Y_{\nu} \cdot ||\Theta {\bf q} - \pmb{\eta}||,
        $$

        which implies 

        \begin{equation}\label{boundforpsi}
        ||\Theta {\bf q} - \pmb{\eta}|| \geq \frac{1}{2n} \frac{F_{\pmb{\eta}}(Y_\nu)}{Y_\nu}.
        \end{equation}

        Applying $G_{\pmb{\eta}}$ to the first inequality in (\ref{boundsforq}), one can see that

        $$
        Y_\nu \leq G_{\pmb{\eta}}(2m|{\bf q}|) \,\,\,\,\,\,\,\,\,\,\,\, \text{and} \,\,\,\,\,\,\,\,\,\,\,\, \frac{F_{\pmb{\eta}}(Y_\nu)}{Y_\nu} \geq \frac{F_{\pmb{\eta}}\left(G_{\pmb{\eta}}(2m|{\bf q}|)\right)}{G_{\pmb{\eta}}(2m|{\bf q}|)} = \frac{1}{\tilde{\rho}_{\pmb{\eta}}(2m|{\bf q}|)}.
        $$

        Therefore, the inequality $||\Theta {\bf q} - \pmb{\eta}|| \cdot \tilde{\rho}(2m|{\bf q}|) \geq \frac{1}{2n} $ holds for all ${\bf q}$ with $|{\bf q}|$ large enough. 

        \item[(b)] By (\ref{liminf11}), there exists infinitely many $\nu$ such that

        \begin{equation}\label{ineqliminf}
            \psi_{\Theta^{\top}} (Y_\nu) \leq \psi(Y_{\nu}).
        \end{equation}
        Let us fix such a $\nu$. Fix ${\bf q}$ such that 
        \begin{equation} \label{boundsforq_b}
        2m|{\bf q}| \leq \frac{F_{\pmb{\eta}}(Y_{\nu})}{\psi(Y_{\nu})}: = 2mQ_{\nu}.
        \end{equation}

        Using (\ref{scalarprodbound}), (\ref{boundsforq_b}) and (\ref{ineqliminf}), we get that 

        $$
        F_{\pmb{\eta}}(Y_\nu) \leq m |{\bf q}| \psi_{\Theta^{\top}} (Y_\nu) + n Y_{\nu} \cdot ||\Theta {\bf q} - \pmb{\eta}|| \leq 
        $$
        $$
         m |{\bf q}| \psi(Y_{\nu}) + n Y_{\nu} \cdot ||\Theta {\bf q} - \pmb{\eta}|| \leq \frac{1}{2} F_{\pmb{\eta}}(Y_\nu) + n Y_{\nu} \cdot ||\Theta {\bf q} - \pmb{\eta}||,
        $$

        which again implies (\ref{boundforpsi}).

        We apply $G_{\pmb{\eta}}$ to (\ref{boundsforq_b}) to see that

        $$
        Y_\nu = G_{\pmb{\eta}}(2mQ_{\nu}) \,\,\,\,\,\,\,\,\,\,\,\, \text{and} \,\,\,\,\,\,\,\,\,\,\,\, \frac{F_{\pmb{\eta}}(Y_\nu)}{Y_\nu} = \frac{F_{\pmb{\eta}}\left(G_{\pmb{\eta}}(2mQ_{\nu})\right)}{G_{\pmb{\eta}}(2mQ_{\nu})} = \frac{1}{\tilde{\rho}_{\pmb{\eta}}(2mQ_{\nu})},
        $$

        which means that the inequality $||\Theta {\bf q} - \pmb{\eta}|| \cdot \tilde{\rho}(2mQ_{\nu}) \geq \frac{1}{2n} $ holds for all ${\bf q}$ with $|{\bf q}| \leq Q_{\nu}$.

         \end{itemize}

\end{proof}

\vskip 0.3 cm

To prove Theorem \ref{Dyakova}, we will need the following statement, which is a special case of the Theorem 4.1 in \cite{BFS}. We denote Euclidean norm by $| \cdot |_2$.

\begin{proposition}\label{thmhaw}
    Let $\{ { \bf z}_\nu \}_{\nu \in \mathbb{N}}$ be a sequence of vectors in $\mathbb{R}^d$ such that the sequence $\{ |{ \bf z}_\nu |_2 \}_{\nu \in \mathbb{N}}$ of their Euclidean norms is lacunary. Let $\{ {\delta}_\nu \}_{\nu \in \mathbb{N}}$ be a sequence of real numbers. Then the set 
    $$
    \{ \pmb{\eta} \in \mathbb{R}^d: \,\,\, \inf\limits_{\nu \in \mathbb{Z}_+} ||\pmb{\eta} \cdot { \bf z}_{\nu} + \delta_{\nu}|| > 0  \}
    $$
    is HAW.
\end{proposition}

\begin{lemma}\label{winninglemma}
    Let $\mathcal{L}^{\perp} \cap \Omega = \emptyset$, where $\mathcal{L}, \mathcal{A}$ and $\Omega$ are as in Theorem \ref{Dyakova}. Let $N: = \{ \pmb{\eta} \in \mathbb{R}^n: \,\,\, \inf\limits_{\nu \in \mathbb{Z}_+} ||\pmb{\eta} \cdot { \bf y}_{\nu}|| > 0  \}$. Then the set $N \cap \mathcal{A}$ is HAW in $\mathcal{A}$.
\end{lemma}

\begin{proof}
    Let $\measuredangle({{\bf v}, \mathcal{L}})$ denote the angle between the vector $\bf v \in \mathbb{R}^n$ and linear subspace $\mathcal{L}$. The set $\Omega$ is compact, so there exists 

    $$
    0 < \varphi_0 = \inf\limits_{v \in \Omega} \left|\measuredangle({{\bf v}, \mathcal{L}}) - \frac{\pi}{2}\right| = \min\limits_{v \in \Omega} \left|\measuredangle({{\bf v}, \mathcal{L}}) - \frac{\pi}{2}\right|.
    $$

    Let $\Phi_k : = \{ v \in \mathbb{R}^n: \,\,\, \left( \frac{2}{3} \right)^k < \cos \left( \measuredangle({v, \mathcal{L}}) \right) \leq \left( \frac{2}{3} \right)^{k-1} \}$. Then, there exists $k_0$ such that $\Omega \subset \bigcup\limits_{k=1}^{k_0} \Phi_k$, or, in other words, for any $\nu$ there exists $k \in \{1, \ldots, k_0  \}$ such that ${\bf y}_{\nu} \in \Phi_k$. Now, let us represent the sequence $\{ {\bf y}_{\nu} \}$ as the union of finitely many subsequences as follows:

    \begin{itemize}
        \item Let ${\bf y}_{\nu}^j: = {\bf y}_{C(\nu-1) + j}, \,\,\,j = 1, \ldots, C,$ where $C$ is such a constant that $\frac{|{\bf y}_{\nu + C}|_2}{|{\bf y}_{\nu}|_2} > 2$; the existence of such $C$ follows from Proposition \ref{exp}, since sup norm is equivalent to Euclidean norm.  Then 
        $$\{ {\bf y}_{\nu} \} = \bigcup\limits_{j=1}^C \{ {\bf y}_{\nu}^j \} \,\,\,\,\,\,\,\,\, \text{and} \,\,\,\,\,\,\,\,\, \frac{|{\bf y}_{\nu+1}^j|_2}{|{\bf y}_{\nu}^j|_2} \geq 2 \,\,\,\,\, \text{ for any} \,\, \nu.$$
        \item Let $\{ {\bf y}_{\nu}^{j, k} \} : = \{ {\bf y}_{\nu}^{j} \} \cap \Phi_k$ (we change the numeration accordingly). Then, 

        $$\{ {\bf y}_{\nu} \} = \bigcup\limits_{j=1}^A \bigcup\limits_{k=1}^{k_0} \{ {\bf y}_{\nu}^{j, k} \},$$
        
        and 

        \begin{equation} \label{lacunary1}
        \frac{\Big| pr_{\mathcal{L}}{\bf y}_{\nu+1}^{j, k}\Big|_2}{\Big|pr_{\mathcal{L}}{\bf y}_{\nu}^{j,k}\Big|_2} = \frac{\Big|{\bf y}_{\nu+1}^{j, k}\Big|_2 \cos \left( \measuredangle({{\bf y}_{\nu+1}^{j, k}, \mathcal{L}}) \right)}{\Big|{\bf y}_{\nu}^{j,k}\Big|_2\cos \left( \measuredangle({{\bf y}_{\nu}^{j, k}, \mathcal{L}}) \right)} \geq \frac{2}{3} \frac{\Big|{\bf y}_{\nu+1}^{j, k}\Big|_2}{\Big|{\bf y}_{\nu}^{j,k}\Big|_2} \geq \frac{4}{3};
        \end{equation}

        in particular, each of the sequences $\{| pr_{\mathcal{L}} {\bf y}_{\nu}^{j, k}|_2 \}$ is lacunary.

        Let us define 

        $$
        N(j, k): = \{ \pmb{\eta} \in \mathbb{R}^n: \,\, \inf\limits_{\nu \geq 1} || \pmb{\eta} \cdot { \bf y}_{\nu}^{j, k}|| > 0 \}.
        $$

        It is easy to see that $N = \bigcap\limits_{j=1}^A \bigcap\limits_{k=1}^{k_0} N(j, k)$, so it is enough to prove that each of the sets $N(j, k) \cap \mathcal{A}$ is HAW. This follows from  (\ref{lacunary1}) and Proposition \ref{thmhaw} by a standard argument:

        \item Fix indices $j$ and $k$. Fix a vector ${\bf a} \in \mathcal{A}$, and let $\delta_{\nu }: = { \bf a } \cdot { \bf y}_{\nu}^{j, k}$. Let ${\bf l}_1, \ldots, {\bf l}_s$ be an orthonormal basis of $\mathcal{L}$, and $L$ a matrix with columns ${\bf l}_i, \,\, i = 1, \ldots, s$. Then, $\pmb{\lambda} \mapsto L \pmb{\lambda} + { \bf a}$ is an affine bijection between $\mathbb{R}^s$ and $\mathcal{A}$. Thus, it is enough to show that the set

        \begin{equation}\label{lambdasethaw}
        \{ \pmb{\lambda} \in \mathbb{R}^s: \,\, \inf\limits_{\nu \geq 1} || \pmb{\lambda} \cdot \left( { \bf y}_{\nu}^{j, k} \right)^{\top} L + \delta_{\nu}|| > 0 \}
        \end{equation}
is HAW.

Since the sequence $\{| pr_{\mathcal{L}} {\bf y}_{\nu}^{j, k}|_2 \}$ is lacunary, so is the sequence $\left\{ \left|\left( { \bf y}_{\nu}^{j, k} \right)^{\top} L \right|_2 \right\}_{\nu \in \mathbb{N}}$. Now HAW property of the set \eqref{lambdasethaw} follows directly from Proposition \ref{thmhaw}.        
    \end{itemize}
\end{proof}

\vskip 0.3 cm

\noindent {\it Proof of Theorem \ref{Dyakova}.}

Let $N$ be as in Lemma \ref{winninglemma} and $F_{\pmb{\eta}}(t) : = \inf\limits_{\nu \in \mathbb{Z}_+} ||\pmb{\eta} \cdot { \bf y}_{\nu}|| = c(\pmb{\eta}) > 0 $. In this case $G_{\pmb{\eta}}(t) = \rho \left( \frac{1}{c(\pmb{\eta})} t
 \right)$ and $\tilde{\rho}_{\pmb{\eta}}(t) = \frac{1}{c(\pmb{\eta})} \rho \left( \frac{1}{c(\pmb{\eta})} t
 \right)$. By Lemma \ref{transference} then $N \subseteq S(\Theta)$, where $K(\pmb{\eta}) = \frac{2m}{c(\pmb{\eta})}$.

 \vskip+0.3cm

 \noindent {\it Proof of Theorem \ref{D_measure}.}

 Let $F_{\pmb{\eta}}(t) : = \omega(t)$. The result follows from Lemma \ref{transference}, Lemma \ref{lemma2bl}, and Lemma \ref{zerodim}.



\subsection{Proof of Theorem \ref{thm10}}\label{thm10proofsect}

First, note that due to Remark \ref{trivsingremark} we can assume that $\Theta^{\top}$ is not trivially singular and hence has a well defined sequence of best approximation vectors.

Let us consider a decreasing function $\psi$
and function $\rho$  defined by (\ref{notationchange}); notice that such $\psi$ and $\rho$ are related via (\ref{xxrho}). Let  
$$\tilde{g}(T) 
 = \left( \frac{\varepsilon}{n} \right)^n g \left( \left( \frac{m}{\varepsilon} \right)^m T \right).
 $$
    
    We know that $\Theta^\top$ is $f$-approximable, or equivalently that the inequality
    \begin{equation}\label{xx00}\psi_{\Theta^\top}(|{\bf y}|) \leq \psi(|{\bf y}|)
     \end{equation}
    holds for infinitely many ${\bf y} \in \mathbb{Z}^n$.
   
Define
 \begin{equation}\label{xx0}
\sigma (t) = \frac{\varepsilon}{n \rho \left( \frac{m}{\varepsilon} \cdot t \right)}. 
 \end{equation}
   Let us notice that 
    $$
    \widehat{\bf D}_{n,m}^\Theta \left[ \tilde{g} \right] \cap \mathcal{A} =
    \left\{ \pmb{\eta} \in \mathcal{A}:\,\,
    \text{inequality}\,\,\,
    \psi_{\Theta, \pmb{\eta}}(t) \leq \sigma (t) 
    \,\,\,
        \text{holds for every} \,\, t \,\, \text{large enough}\right\}.
        $$
 
 We should  also note that as $ \psi(t) $ decreases to zero as $ t\to \infty$, function $\sigma(t) $ decreases.

 {
 Consider the infinite increasing sequence of integer
 vectors ${\bf y}\in \mathbb{Z}^n$
  for which (\ref{xx00}) holds.
We may assume that this sequence is a subsequence of the best approximation vectors for matrix $\Theta^\top$. To avoid cumbersome notation, in this proof  we do not denote this subsequence
as
${\bf y}_{\nu_k}, k=1,2,3,...$,
but
simply as
${\bf y}_\nu, \nu=1,2,3,...$ and do not use double indices.
  Let 
 $Y_\nu =|{\bf y}_\nu|$.
Let
  \begin{equation}\label{xx2}
t_\nu = \frac{\varepsilon}{m\psi(Y_\nu)}.
 \end{equation}
 It is clear that $t_\nu\to \infty$ when $ \nu \to \infty$. Therefore, for $\pmb{\eta} \in\widehat{\bf D}_{n,m}^\Theta \left[ \tilde{g} \right] \cap \mathcal{A} $ and for large enough $\nu$ there exists ${\bf q} \in \mathbb{Z}^m$ such that $|{\bf q}| \leq t_{\nu}$ and $||\Theta{\bf q} - \pmb{\eta}|| \leq \sigma(t_{\nu})$.

Now by (\ref{scalarprodbound}), one has 
 \begin{equation}\label{xx3}
    || \pmb{\eta} \cdot { \bf y}_{\nu}||
        \leq  m |{\bf q}| \psi_{\Theta^{\top}} (Y_\nu) + n Y_{\nu} ||\Theta{\bf q} - \pmb{\eta}||\leq
       m t_{\nu} \psi(Y_\nu) + n Y_\nu  \sigma (t_{\nu})
     \end{equation}
        for any $\nu$ large enough. 
        
      Note that   by (\ref{xx2}) we have 
        $$
            m t_{\nu} \psi(Y_\nu) = \varepsilon,
        $$
         meanwhile   
        by (\ref{xxrho}) and (\ref{xx0})
         we have
        $$
           n Y_\nu  \sigma (t_{\nu})  = \frac{\varepsilon Y_\nu}{\rho\left(\frac{1}{\psi (Y_\nu)}\right)} = \varepsilon.
        $$
          Now the last two equalities by
 (\ref{xx3}) imply 
 \begin{equation}\label{xx6}
 \pmb{\eta} \in\widehat{\bf D}_{n,m}^\Theta \left[ \tilde{g} \right] \cap \mathcal{A}\,\,\,\,
 \Longrightarrow \,\,\,\, 
    || \pmb{\eta} \cdot { \bf y}_{\nu}||
        \leq  2\varepsilon.
     \end{equation}

Fix ${\bf a} \in \mathcal{A}$ and define

 $$
\Omega_{\nu } = \Omega_{\nu}(\mathcal{A}, M,{\bf a},  \varepsilon) = \{ \pmb{\eta} \in \mathcal{A} \cap \left(B_M(0) + {\bf a}\right): \,\,\, || \pmb{\eta} \cdot { \bf y}_{\nu}||
        \leq  2\varepsilon
\}.
  $$          

From (\ref{xx6}) we see that }
\begin{equation}\label{xx8}
 \widehat{\bf D}_{n,m}^\Theta \left[ \tilde{g} \right] \cap \mathcal{A} \cap \left(B_M(0) + {\bf a}\right) \subseteq \bigcup\limits_{N=1}^{\infty} \bigcap\limits_{\nu = N}^{\infty} \Omega_{\nu}
\end{equation}

(recall that now  $B_M(0)+{\bf a}$ is a $({\rm dim}\, \mathcal{A})$-dimensional ball  of radius $M$ in sup-norm 
in $\mathcal{A}$ centered at $ {\bf a}\in \mathcal{A}$),
so it is enough to prove that for any fixed $M$ and ${\bf a}$ the latter set has measure zero under given restrictions on $\varepsilon$.
Further, by (\ref{xx8}) we have
$$
\left(B_M(0) + {\bf a}\right)\setminus 
\widehat{\bf D}_{n,m}^\Theta[\tilde{g}] 
\supset \left(B_M(0) + {\bf a}\right)\setminus\left(
\bigcup\limits_{N=1}^{\infty} \bigcap\limits_{\nu = N}^{\infty} \Omega_{\nu}
\right)=
\bigcap\limits_{N=1}^{\infty} \bigcup\limits_{\nu = N}^{\infty} \Omega_{\nu}^c,\,\,\,
\text{where}\,\,\,
\Omega_{\nu}^c = \left(B_M(0) + {\bf a}\right) \setminus \Omega_{\nu}.
$$

        Now, we should note that  by the assumption of Theorem \ref{thm10}, $\mathcal{A}$ is non-$\Xi$-exceptional for some $\Xi < 1$. This means that 
        (\ref{xxExe}) is not satisfied an so 
      $|{\rm pr}_{\mathcal{L}}\,{\bf y}_{\nu}| \rightarrow \infty$. 
      By Lemma \ref{lemmameasures} we see that
       \begin{equation}\label{xx7} 
        m({\Omega}_{\nu}) \leq  4 \varepsilon (1 + \xi(M)) \cdot m(B_M(0)),
        \end{equation}
        where $\lim\limits_{M \rightarrow \infty} \xi(M) = 0$. 
         As $0 < \varepsilon < \frac{1}{4}$, by (\ref{xx7}) we can take  $M$ large enough to ensure the inequality $4\varepsilon(1 + \xi(M)) < 1$. Now 
  $$ m({\Omega}_{\nu}^c) \geq \left( 1 - 4\varepsilon (1 + \xi(M)) \right) m(B_M(0)) > C m(B_M(0))
  $$ for all $\nu$ large enough
  with some positive $C$,
  and so 
  \begin{equation}\label{xx9}
  \sum\limits_{k=1}^{\infty} m({\Omega}_{\nu}^c) = \infty.
    \end{equation}
        
      As we have noticed in the beginning of our proof, we may take a subsequence of the sequence of all the best approximations.
      Now it is necessary to take even more sparse subsequence to satisfy the condition 
       $$
      \lim\limits_{k \rightarrow \infty} \frac{|{\rm pr}_{\mathcal{L}}\,{\bf y}_{\nu}|}{|{\rm pr}_{\mathcal{L}}\,{\bf y}_{\nu+1}|} = 0.
      $$
      Under this condition we can apply Lemma \ref{independencelemma}.
      Together with 
 Lemma \ref{Sprin}  and (\ref{xx9}) this gives 
        $$
        m \left(  \bigcap\limits_{N=1}^{\infty} \bigcup\limits_{\nu=N}^{\infty} {\Omega}_{\nu}^c \right) \geq \limsup\limits_{N \rightarrow \infty} \frac{\left( \sum\limits_{\nu=1}^{N} m({\Omega}_{\nu}^c) \right)^2}{\sum\limits_{\nu,\mu=1}^{N} m({\Omega}_{\nu}^c \cap {\Omega}_{\mu}^c)} = m(B_M(0)),
        $$
        and so 
        $$
        m \left( \widehat{\bf D}_{n,m}^\Theta \left[ \tilde{g}\right] \cap \left(B_M(0) + {\bf a}\right) \right) \leq m \left(  \bigcup\limits_{N=1}^{\infty} \bigcap\limits_{\nu=N}^{\infty} {\Omega}_{\nu} \right) = m(B_M(0)) - m(B_M(0)) = 0,
        $$
which completes the proof.

\subsection{A metric version of Cassels' lemma. Proof of Theorem \ref{nonSingtoBA_measure}}\label{BAnullsection}
    
In this section we will prove a  metric analog of Lemma \ref{lemma3bl} and obtain Theorem \ref{nonSingtoBA_measure} as a corollary.

    \begin{lemma} \label{boxeslemma}

    Let $\{\psi_k\}$ be such a sequence of real numbers that $0 < \psi_k < 1$ for any $k$, $\lim\limits_{k \rightarrow \infty} \psi_k = 0$ and the series $\sum\limits_{k=1}^{\infty} \psi_k^n$ diverges.

        Let $\{ M_k\}, \, \{X_k\}$ be two sequences of real numbers such that the inequality

    $$
    ||\Theta^{\top} {\bf y}|| \geq \frac{\kappa}{X_k} 
    $$

    holds for any nonzero ${\bf y} \in \mathbb{Z}^n$ with $|{\bf y}| < M_k$, and in addition $\lim\limits_{k \rightarrow \infty} \psi_k \frac{M_{k+1}}{M_k} = \infty$.
    
    Then, for almost every $\pmb{\eta} \in \mathbb{R}^n$ there exists infinitely many $k \in \mathbb{N}$, such that the system of inequalities

    \begin{equation}
        \begin{cases}
            |{\bf q}| \leq X_k \\
            ||\Theta {\bf q} - \pmb{\eta}|| \leq \frac{\kappa \psi_k}{M_k}
        \end{cases}
    \end{equation}

    has a solution ${\bf q} \in \mathbb{Z}^m$.
    \end{lemma}

    \begin{proof} Fix $\varepsilon > 0$.

        Fix $k$, and let 
        $$
        B_k(\pmb{\eta}): = \left[ \eta_1 - \frac{\kappa}{M_k}, \eta_1 + \frac{\kappa}{M_k}\right) \times \ldots \times \left[ \eta_n - \frac{\kappa}{M_k}, \eta_n + \frac{\kappa}{M_k}\right)
        $$

        where $\pmb{\eta} = (\eta_1, \ldots, \eta_n)^T$. Let $\pmb{\eta}_{i_1, \ldots, i_n} = \left( \frac{2 \kappa i_1}{M_k}, \ldots, \frac{2 \kappa i_n}{M_k} \right)$; then the cube $[0, 1)^n$ can be covered by $W_k$ boxes of the form $B_k(\pmb{\eta}_{i_1, \ldots, i_n})$ where $W_k = \lceil \frac{M_k}{2 \kappa} \rceil^n$ or $\left( \lceil \frac{M_k}{2 \kappa} \rceil +1 \right)^n$.

        Let us also consider boxes of the form 

        $$
        I_k({\bf q}): = \left[ \sum\limits_{j=1}^m \theta_{1,j} q_j - \frac{\kappa \psi_k}{M_k}, \sum\limits_{j=1}^m \theta_{1,j} q_j + \frac{\kappa \psi_k}{M_k}\right] \times \ldots \times \left[ \sum\limits_{j=1}^m \theta_{n,j} q_j - \frac{\kappa \psi_k}{M_k}, \sum\limits_{j=1}^m \theta_{n,j} q_j + \frac{\kappa \psi_k}{M_k}\right]
        $$

        for $|{\bf q}| \leq X_k$. Lemma \ref{lemma3bl} guarantees that any box $B_k(\pmb{\eta}_{i_1, \ldots, i_n})$ from our cover contains at least one point of the form $\Theta {\bf q}$ with $|{\bf q}| \leq X_k$. Let us pick one such point ${\bf q}$ for each collection $(i_1, \ldots, i_n)$ such that $i_1 \equiv \ldots \equiv i_n \equiv 1 \,\,$ (mod 2). We will denote these points by ${\bf q}^l(k)$, where $1 \leq l \leq W_k'$ and $W_k'$ is the number of the boxes $B_k(\pmb{\eta}_{i_1, \ldots, i_n})$ belonging to $[0, 1)^n$ and satisfying $i_1 \equiv \ldots \equiv i_n \equiv 1 \,\,$ (mod 2). One can see that for $k$ large enough, $\sqrt{1 - \varepsilon} \left(\frac{M_k}{4 \kappa} \right)^n \leq W_k' \leq  \sqrt{1 + \varepsilon} \left(\frac{M_k}{4 \kappa} \right)^n$, and in particular $W_k' \asymp W_k \asymp M_k^n$.

        Let us fix $s < k$ and $1 \leq l \leq W_s'$; then, 

        \begin{equation}\label{numberofpoints}
        \# \{ r: \,\, 1 \leq r \leq W_k', \,\, I_s({\bf q}^l(s)) \cap I_k({\bf q}^r(k)) \neq \emptyset \} \leq \Big\lceil \frac{2 \kappa \psi_s \cdot M_k}{2 \cdot M_s \cdot 2 \kappa} + \frac{1}{2} \Big\rceil \leq (1 + \varepsilon) W_k' m\left( I_s({\bf q}^l(s)) \right)
        \end{equation}

        (here the last inequality also requires us to choose large enough $k$).

        Let $E_k: = \bigcup\limits_{r=1}^{W_k'} I_k({\bf q}^r(k))$. We know that this union is disjoint, and so 

        $$
        m(E_k) = \sum\limits_{r=1}^{W_k'} \left( \frac{2\kappa \psi_k}{M_k} \right)^n  \gg \psi_k^n,
        $$

        which implies that $\sum\limits_{k=1}^{\infty} m(E_k) = \infty$. We also know by (\ref{numberofpoints}) that for any $s < k$,

        $$
        m(E_s \cap E_k) \leq (1 + \varepsilon) \sum\limits_{l = 1}^{W_s} \sum\limits_{r = 1}^{W_k} W_s' W_k' m\left( I_s({\bf q}^l(s)) \right) m\left( I_k({\bf q}^r(k)) \right) \leq
        $$
        $$
        (1 + \varepsilon) \cdot m(E_s) m(E_k) .
        $$

Let $E : = \limsup\limits_{k \rightarrow \infty} E_k$; then by Lemma \ref{Sprin}, $m(E) \geq \frac{1}{1 + \varepsilon}$. It remains to notice that our choice of $\varepsilon$ was arbitrary, which completes the proof.

        \end{proof}

        The next statement directly implies Theorem \ref{nonSingtoBA_measure} via the identification (\ref{notationchange}).

        \begin{corollary}
            Suppose
            
            \begin{equation}\label{thm4ineq1}
                \text{the inequality} \,\,\, \psi_{\Theta^{\top}}(t) > \psi(t) \,\, \,\text{holds for an unbounded set of $t \in \mathbb{R}_+$.}
            \end{equation} 

             Then, for almost all $\pmb{\eta} \in \mathbb{R}^n$

            $$
        \liminf\limits_{t \rightarrow \infty} \psi_{\Theta, \pmb{\eta}}(t) \cdot \rho \left( \frac{1}{\kappa}t \right) = 0.
        $$
        \end{corollary}

        \begin{proof}
        First of all, let us notice that the condition (\ref{thm4ineq1}) implies that $\Theta^\top$ is not trivially singular (otherwise, $\psi_{\Theta^{\top}}(t)$ would become zero for all large enough $t$), so the sequence of the best approximation vectors for $\Theta^{\top}$ is well defined. Take any $\{ \psi_k \}$ as in Lemma \ref{boxeslemma}. Let $\{ \nu_k \}$ be a sequence of natural numbers, such that $\psi_{\Theta^{\top}}(|{\bf y}_{\nu_k}|) \geq  \psi(Y_{\nu_{k}+1})$, and sparse enough to guarantee $\psi_k \frac{Y_{\nu_{k}+1}}{Y_{\nu_{k}}} \rightarrow \infty$.
            
            Let $M_k: = Y_{\nu_{k}+1}$ and $X_k: = \frac{\kappa}{||\Theta^{\top} {\bf y}_{\nu_{k}}||}$. For almost every $\pmb{\eta} \in \mathbb{R}^n$ there exists infinitely many $k \in \mathbb{N}$, such that the system of inequalities

    $$
        \begin{cases}
            |{\bf q}| \leq \frac{\kappa}{\psi_{\Theta^{\top}}(|{\bf y}_{\nu_k}|)}  \leq \frac{\kappa}{\psi(Y_{\nu_{k}+1})}\\
            ||\Theta {\bf q} - \pmb{\eta}||Y_{\nu_{k}+1} \leq \kappa \psi_k
        \end{cases}
    $$

    has a solution ${\bf q} \in \mathbb{Z}^m$. Applying $\rho$ to the first inequality, one can see that

    $$
    \rho \left(  \frac{1}{\kappa} |{\bf q}| \right) \leq Y_{\nu_{k}+1}
    $$

    and so

    $$
    \psi_{\Theta, \pmb{\eta}}(|{\bf q}|) \cdot \rho \left( \frac{1}{\kappa}|{\bf q}| \right) \leq ||\Theta {\bf q} - \pmb{\eta}|| \rho \left(  \frac{1}{\kappa} |{\bf q}| \right)  \leq ||\Theta {\bf q} - \pmb{\eta}||Y_{\nu_{k}+1} \leq \kappa \psi_k
    $$

    which implies the desired statement.
\end{proof}

\section{Proof of Theorem \ref{measure_corollary} and Theorem \ref{mainthm_asympt} \ref{part3meas}}\label{deductionsection}

We want to deduce Theorem \ref{measure_corollary} and Theorem \ref{mainthm_asympt} \ref{part3meas} from Theorem \ref{fullme}. We will only show the deduction of Theorem \ref{measure_corollary}; the proof of Theorem \ref{mainthm_asympt} \ref{part3meas} is purely technical and follows the same lines.

First, note that $\lambda(T) : = \left(\frac{u(T)}{g(T)}\right)^{\frac{m \gamma}{m + \gamma n}}$. This implies that condition (\ref{lambda_converges}) holds.

We recall that $\alpha, \beta$ and $\gamma$ are defined in (\ref{technical}),  (\ref{lambdatech}),  (\ref{additionalf}) and (\ref{finiteexp}). Properties (\ref{lambdatech}) and (\ref{additionalf}) imply that condition (\ref{monot_cond}) holds: suppose $\delta > 1$ and $T \in \mathbb{R}$, then
$$
\frac{\lambda \left (T^{\delta} \right)}{f \left (T^{\delta} \right)}  \frac{f \left (T \right)}{\lambda \left (T \right)} > \left(  \frac{\log \left( T^{\delta} \right)}{\log \left( T \right)}
\right)^{\alpha} \delta^{- \alpha} = 1.
$$

Thus, the function $\tilde{g}$, defined by (\ref{lltildeg}) using this $\lambda$, satisfies Theorem \ref{fullme}.

    Let us show that $u \ll \tilde{g}$. We will define $H(T)$ as in Theorem \ref{fullme}, and let $R(T)$ be the inverse function to $\frac{1}{f}$. We note that 
    \begin{equation}\label{g_altern}
        g(T) = \frac{1}{R(T)}.
    \end{equation}
    
    Property (\ref{technical}) implies that 
    \begin{equation}\label{Rineq}
        R(KU) \leq K^{\frac{1}{\gamma}}R(U) \,\,\,\,\,\,\,\,\,\,\,\,\, \text{for any} \,\, U>0 \,\, \text{and any} \,\, K \,\, \text{large enough.}
    \end{equation}

    Applying (\ref{Rineq}), we see that for $U$ large enough one has
    
    \begin{equation}\label{HandR}
    H \left( \frac{\lambda(U)}{f(U)} \right) = U = R \left( \frac{1}{f(U)} \right) = R \left( \frac{1}{\lambda(U)} \frac{\lambda(U)}{f(U)} \right) \leq \frac{1}{\lambda(U)^{\frac{1}{\gamma}}} R \left( \frac{\lambda(U)}{f(U)} \right).
    \end{equation}

    Let us notice that conditions (\ref{technical}) and (\ref{finiteexp}) on $f$ and (\ref{lambdatech}) on $\lambda$ imply that

    \begin{equation}\label{asymp_lambda1}
    \lambda \left( \frac{\lambda(T)}{f(T)} \right) \asymp \lambda(T),
    \end{equation}

    where the implied constant only depends on $\alpha$, $\beta$ and $\gamma$. Combining this with (\ref{HandR}), we see that 

    $$
    H \left( \frac{\lambda(U)}{f(U)} \right) \ll {\lambda\left( \frac{\lambda(U)}{f(U)} \right)^{-\frac{1}{\gamma}}} R \left( \frac{\lambda(U)}{f(U)} \right),
    $$

    which (as the function $\frac{\lambda(U)}{f(U)}$ attains all large enough values) shows that

    \begin{equation}\label{HleqR}
    H(T) \ll {\lambda\left(T \right)^{-\frac{1}{\gamma}}} R \left( T \right) \,\,\,\,\,\,\,\,\,\,\,\, \text{for all} \,\, T \,\, \text{large enough.}
    \end{equation}

    Finally, note that in a similar fashion to (\ref{asymp_lambda1}) one can check that 

    \begin{equation}\label{asymp_lambda2}
    \lambda \left( H(T) \right) \asymp \lambda(T),
    \end{equation}

    where the implied constant again depends only on $\alpha$, $\beta$ and $\gamma$.

    We conclude that

    $$
    u(T) =  \lambda(T)^{\frac{m + \gamma n}{m \gamma}}g(T) \underset{\mathrm{by (\ref{g_altern})}}{=\joinrel=\joinrel=}  \frac{\lambda(T)^{\frac{n}{m}}}{\lambda(T)^{-\frac{1}{\gamma}} R(T)} \underset{\mathrm{by (\ref{HleqR})}}{\ll} \frac{\lambda(T)^{\frac{n}{m}}}{H(T)} \underset{\mathrm{by (\ref{asymp_lambda2})}}{\asymp} \frac{\lambda^{\frac{n}{m}} \left( 
H(T) \right)}{H(T)} = \tilde{g}(T),
    $$

    or simply $u(T) < C \tilde{g}$  for some constant $C = C(\alpha, \beta, \gamma) > 0$.

    By Theorem \ref{fullme} (applied with $\varepsilon = \frac{1}{2}$), this implies that  
    
    \begin{equation}\label{Wwithconst}
        \text{the set} \,\, \widehat{\bf W}_{n,m}^{\Theta} \left[\frac{1}{4n (2m)^n C} u\right] \cap \mathcal{A} \subseteq  \widehat{\bf W}^{\Theta}_{n,m} \left[ \frac{1}{4n(2m)^n } \tilde{g} \right] \cap \mathcal{A} \,\, \text{has measure zero in} \,\, \mathcal{A}.
    \end{equation}     

    It remains to notice that the function $\lambda$ defined using the function ${4n (2m)^n C} \cdot u(T)$ instead of $u(T)$ satisfies conditions (\ref{lambdatech}) and (\ref{lambdaseries}), thus (\ref{Wwithconst}) holds for $u$ replaced with ${4n (2m)^n C} \cdot  u(T)$. This completes the proof.

\section{ On Khintchine's result: Theorem \ref{x2} and Theorem \ref{Khintchine_general}} \label{kkh}

Theorem B follows from the following statement, by taking $h(T) = \phi^n \left( T^{\frac{1}{m}} \right)$. Statements \ref{Kh_partA} and \ref{Kh_partB} of Theorem \ref{Khintchine_general} follow form the corresponding statements of the theorem below.

\begin{theorem} \label{main}
    Let $\phi : \,\, [1; + \infty) \rightarrow \mathbb{R}_+$ be a non-increasing function.

\begin{enumerate}[label=(\alph*)]
  \item\label{st1}  Let $\Theta$ be a real $m \times n$ matrix. For any $\pmb{\eta} \in \mathbb{R}^n$ such that the pair $(\Theta,\pmb{\eta})$ is not trivially singular,  one has
  \begin{equation} \label{inequality0}
    \liminf_{t\to \infty} \frac{\psi_{\Theta} (t)}{\phi(t)}
    \le 2 \limsup_{t\to \infty} \frac{\psi_{\Theta,\pmb{\eta}} (t)}{\phi(2t)}
    .
\end{equation} 
 
    \item\label{st2} Let $\Theta$ be a real $m \times n$ matrix. Then, there exists $\pmb{\eta} \in \mathbb{R}^n$ such that 

\begin{equation} \label{inequality}
    \limsup_{t\to \infty} \frac{\psi_{\Theta,\pmb{\eta}} (t)}{\phi(t)} \le 2B \cdot \liminf_{t\to \infty} \frac{\psi_{\Theta} (t)}{\phi(2A \cdot t)},
\end{equation}

where $A$ and $B$ are defined in (\ref{ab}) in Proposition \ref{exp}. 

If in addition $\Theta$ is not trivially singular, then the set of such $\pmb{\eta}$ is uncountable.
\end{enumerate}

\end{theorem}

\begin{proof}
First we prove statement \ref{st1} of Theorem \ref{main}.
Let $K = \limsup_{t\to \infty} \frac{\psi_{\Theta,\pmb{\eta}} (t)}{\phi( 2 t)} $. Then for any $\varepsilon>0$ one has
$$
\frac{\psi_{\Theta,\pmb{\eta}} (t)}{\phi( 2 t)}\le K+\varepsilon
$$
for  $t$ large enough.

If the pair $(\Theta,\pmb{\eta})$ is not trivially singular, there exists a sequence
$${\bf q}_\nu \in \mathbb{Z}^m, \,\,\, Q_\nu = |{\bf q}_\nu |, \,\,\,\, Q_1=1<Q_2<...<Q_\nu<Q_{\nu+1}<... ,
$$
of  the best inhomogeneous approximations:

$$ \psi_{\Theta,\pmb{\eta}} (t) = \psi_{\Theta,\pmb{\eta}} (Q_\nu) =||\Theta {\bf q}_\nu -\pmb{\eta}||_{\mathbb {Z}^n},\,\,\, \,\,\,\, \text{for}\,\,\,\, Q_\nu \le t <Q_{\nu+1}.
$$ 
Then the inequality 

\begin{equation}\label{1}
\psi_{\Theta,\pmb{\eta}} (Q_\nu) = ||\Theta {\bf q}_\nu-\pmb{\eta} ||_{\mathbb{Z}^n} \le (K+\varepsilon) 
\phi(2Q_{\nu+1})
\end{equation}
holds  for all $\nu$ large enough.

Let ${\bf x}_\nu = {\bf q}_{n+1}- {\bf q}_\nu\in \mathbb{Z}^m$.
Then $|{\bf x}_\nu|\le 2Q_{\nu+1}.$
From the triangle inequality and (\ref{1}) it follows that the inequality
$$
\psi_{\Theta} (|{\bf x}_\nu|)\le 
||\Theta {\bf x}_\nu ||_{\mathbb{Z}^n}\le
\psi_{\Theta,\pmb{\eta}} (Q_\nu)+ \psi_{\Theta,\pmb{\eta}} (Q_{\nu+1})
\le 2 \psi_{\Theta,\pmb{\eta}} (Q_\nu)\le 2(K+\varepsilon) 
\phi(2Q_{\nu+1}) \le 2(K+\varepsilon) 
\phi(|{\bf x}_\nu|) 
$$
holds for $\nu$ large enough. We see that
$$
   \liminf_{t\to \infty} \frac{\psi_{\Theta} (t)}{\phi(t)}\le  \liminf_{\nu\to \infty} \frac{\psi_{\Theta} (|{\bf x}_\nu|)}{\phi(|{\bf x}_\nu|)}\le 2(K+\varepsilon),
$$
and (\ref{inequality0}) is proven.

 \vskip+0.3cm
Now we prove \ref{st2}.

First, suppose $\Theta$ is trivially singular. Then, one can take $\pmb{\eta} = \Theta {\bf p}$ for some   ${\bf p} \in \mathbb{Z}^m$; in this case both left and right hand sides of (\ref{inequality}) are equal to zero, and the desired statement holds.

Now let us suppose that $\Theta$ is not trivially singular, and consider the sequence of best approximations as in Section \ref{bestapproximations}. Let $L = \liminf_{t\to \infty} \frac{\psi_{\Theta} (t)}{\phi(2A \cdot t)}$, and let $\{ t_k \}$ be a sequence such that 

\begin{enumerate}
    \item[(i)] For each $k$, there exists $\nu$ such that $t_k < P_{\nu} \leq t_{k+1}$;
    \item[(ii)] $\lim_{k\to \infty} \frac{\psi_{\Theta} (t_k)}{\phi(2A \cdot t_k)} = L$.
\end{enumerate}

We define $ \nu_k \in \mathbb{Z}_+,\, k \in \mathbb{N},$ by $P_{\nu_k} \leq t_k < P_{\nu_k + 1}$. Then $\psi_{\Theta} (t_k) = \psi_{\Theta} (P_{\nu_k})$ and $\phi(2A \cdot t_k) \leq \phi(2A \cdot P_{\nu_k})$, which implies

$$
\frac{\psi_{\Theta} (P_{\nu_k})}{\phi(2A \cdot P_{\nu_k})} \leq \frac{\psi_{\Theta} (t_k)}{\phi(2A \cdot t_k)} \,\,\,\,\,\,\,\,\,\,\,\, \text{and} \,\,\,\,\,\,\,\,\,\,\,\, \lim_{k\to \infty} \frac{\psi_{\Theta} (P_{\nu_k})}{\phi(2A \cdot P_{\nu_k})} = L.
$$

We consider the vector
 $$
  \pmb{\eta} = \sum_{k=1}^\infty  (\Theta {\bf p}_{\nu_k} - {\bf a}_{\nu_k}) \in \mathbb{R}^n.
  $$
  
  We define $ {\bf y}_l = \sum_{k=1}^l {\bf p}_{\nu_k} \in \mathbb{Z}^m$. 
  
  One can see that
  $$
  ||\Theta  {\bf y}_l -\pmb{\eta}||_{\mathbb{Z}^n} = \left|\left| \sum_{k=l+1}^\infty  (\Theta {\bf p}_{\nu_k} - {\bf a}_{\nu_k})\right|\right|_{\mathbb{Z}^n}.
  $$
    
By Proposition \ref{exp}, we have that
\begin{equation}
 |{\bf y}_l |\le A \cdot P_{\nu_l} + A \cdot \frac{ P_{\nu_l}}{2} + \ldots \le 2A \cdot  P_{\nu_l}, 
 \end{equation}

 and for any $\varepsilon > 0$,
 
 \begin{equation}\label{3}
 \frac{||\Theta  {\bf y}_l -\pmb{\eta}||_{\mathbb{Z}^n} }{\phi(|{\bf y}_{l+1}|)} \le \frac{2B \psi_{\Theta}  (P_{\nu_{l+1}})}{\phi(|{\bf y}_{l+1}|)} \leq 2B \frac{\psi_{\Theta}  (P_{\nu_{l+1}})}{\phi(2A \cdot P_{\nu_{l+1}})} \leq 2B \left(L + \varepsilon \right)
 \end{equation}
 for any $l \ge l(\varepsilon)$.

Let us take $t \geq |{\bf y}_{l(\varepsilon)}|$. As $|{\bf y}_l| \rightarrow \infty$, there exists  $l$ such that $ |{\bf y}_l|\le t < | {\bf y}_{l+1}| $. From the monotonicity of the function $\psi_{\Theta, \pmb{\eta}} (t)$ and ({\ref{3}}) it follows   that 
$$
\frac{\psi_{\Theta,\pmb{\eta}} (t)}{\phi(t)} \le \frac{\psi_{\Theta,\pmb{\eta}} (|{\bf y}_{l}|) }{\phi(|{\bf y}_{l+1}|)}  \leq 2B \left(L + \varepsilon \right).
$$
We see that  the inequality 
$$
\frac{\psi_{\Theta,\pmb{\eta}} (t)}{\phi(t)} \leq 2B \left(L + \varepsilon \right)
$$
holds  for all $t$ large enough. So

$$
\limsup_{t\to \infty} \frac{\psi_{\Theta,\pmb{\eta}} (t)}{\phi(t)} \leq 2B \cdot L = 2B \cdot \liminf_{t\to \infty} \frac{\psi_{\Theta} (t)}{\phi(2A \cdot t)}.
$$

For the statement about uncountability, let us  notice that

\begin{enumerate}
    \item[(i)] There exists uncountably many such sequences $\{ \nu_k \}$ with the additional condition that $\nu_{k+1} - \nu_k \geq 2B$ and
    \item[(ii)] If vectors (\ref{vect})  are linearly independent over $\mathbb{Q}$, then all the terms $\Theta {\bf p}_{\nu_k} - {\bf a}_{\nu_k}$ are nonzero,
\end{enumerate}

which implies that $\pmb{\eta}$ generated by different sequences are different.

\end{proof}

\begin{remark}
    The same argument proves Theorem \ref{main} in a slightly different form. Namely,
    
    {\rm 1.} If $\frac{\psi_{\Theta, \pmb{\eta}}(t)}{\phi(2t)} \leq K$ for any $t$ large enough, then the inequality $\frac{\psi_{\Theta}(t)}{\phi(t)} \leq 2K$ holds for an unbounded set of $t \in \mathbb{R}$;

    {\rm 2.} If $\frac{\psi_{\Theta}(t)}{\phi(2A \cdot t)} \leq L$ for an unbounded set of $t$, then there exists $\pmb{\eta} \in \mathbb{R}^n$ such that  $\frac{\psi_{\Theta, \pmb{\eta}}(t)}{\phi(t)} \leq 2B\cdot L$ for any $t$ large enough.  

    \vskip+0.3cm

    This version directly implies Theorem \ref{Khintchine_general}.
\end{remark}

\end{document}